\documentclass[10pt, leqno]{article} %%ドキュメントクラス
\usepackage{amsmath, amssymb, amsthm, amscd, ascmac, bm, enumerate, paralist}
\usepackage{hyperref}
\usepackage{graphicx, xcolor}
\usepackage[all]{xy}
\numberwithin{equation}{section}	
\theoremstyle{definition}
	\newtheorem{Def}{Definition}[section]
	\newtheorem{Rem}[Def]{Remark}

\theoremstyle{theorem}
	\newtheorem{Thm}[Def]{Theorem}

	\newtheorem{Prop}[Def]{Proposition}
\newcommand{\R}{\mathbb{R}}
\newcommand{\C}{\mathbb{C}}
\newcommand{\Z}{\mathbb{Z}}

\renewcommand{\L}{\mathbb{L}}
\newcommand{\Rs}{\overline{\mathbb{C}}}

\renewcommand{\Re}{\operatorname{Re}}

\newcommand{\3}{{\rm{I\hspace{-.1em}I\hspace{-.1em}I}}}

\newcommand{\bash}{\backslash}

\newcommand{\ii}[2]{\langle {#1}, {#2}  \rangle}

\renewcommand{\O}{\mathcal{O}}
\title{Higher genus nonorientable maximal surfaces in the Lorentz-Minkowski 3-space}

\author
 {Shoichi Fujimori \thanks{The first author was supported in part by Grant-in-Aid for Scientific Research (C) No. 21K03226 from Japan Society for the Promotion of Science.
 }
 \and 
  Shin Kaneda
 }

\date{September 9, 2021}

\begin{document}
\maketitle

\begin{abstract}
We study nonorientable maximal surfaces in Lorentz-Minkowski 3-space.
We prove some existence results for surfaces of this kind with high genus and one end.
\end{abstract}
%\newpage
\section*{Introduction}
The first example of a complete nonorientable minimal surface in Euclidean 3-space $\R^3$ was found by W. H. Meeks $\3$ in 1981.
He constructed a minimal M\"{o}bius strip \cite{M}(See Figure \ref{fig:nonorimin}:Left). 
Then, in 1993, F. J. L\'{o}pez constructed a minimal once-punctured Klein bottle \cite{L}(See Figure \ref{fig:nonorimin}:Right). 

\begin{figure}[htbp]
\centering
	\begin{tabular}{cc}	
		\begin{minipage}{0.45\hsize}
			\centering
			\includegraphics[width=3.7cm]{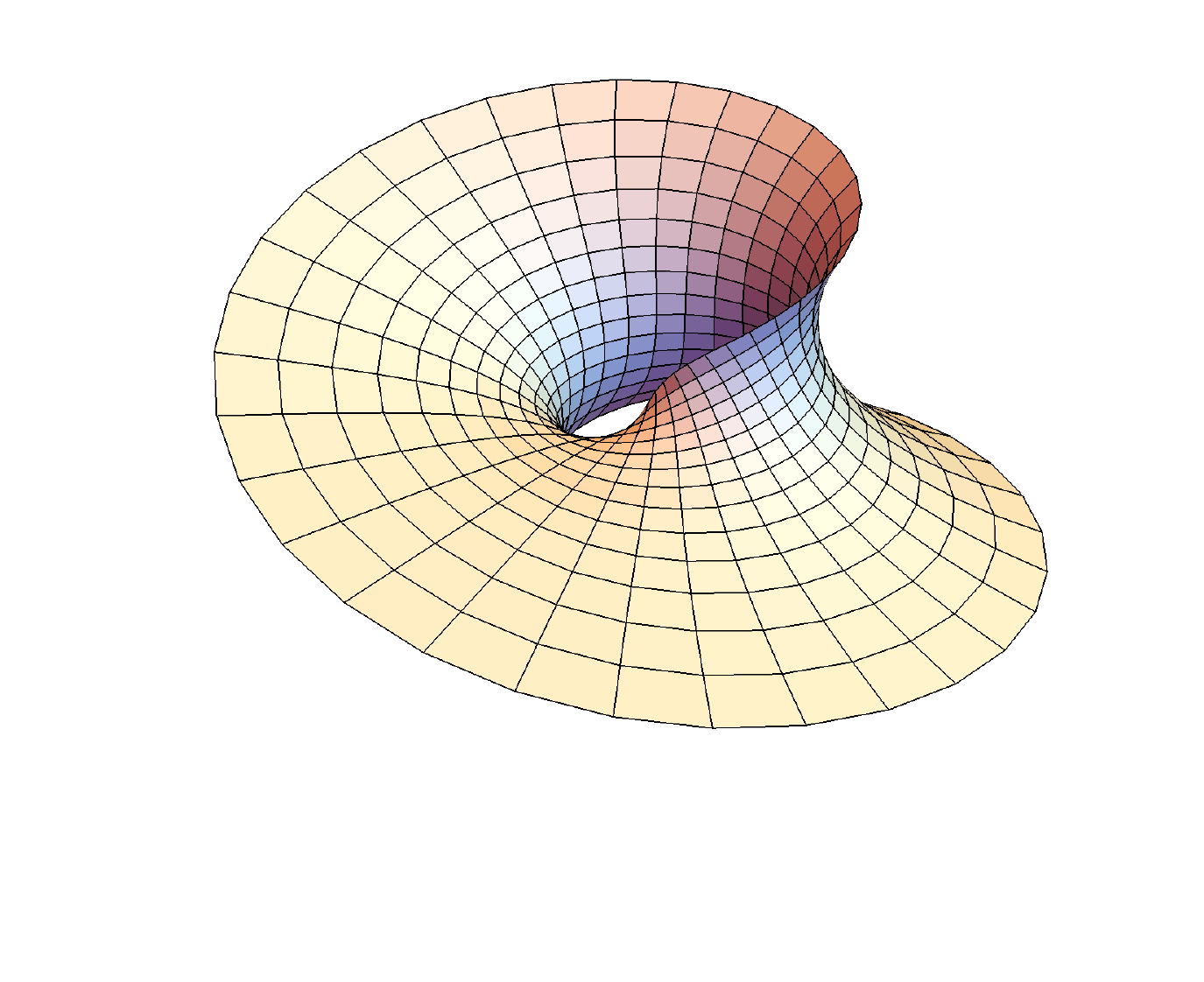}
		\end{minipage}
		
		\begin{minipage}{0.45\hsize}
			\centering
			\includegraphics[width=3.5cm]{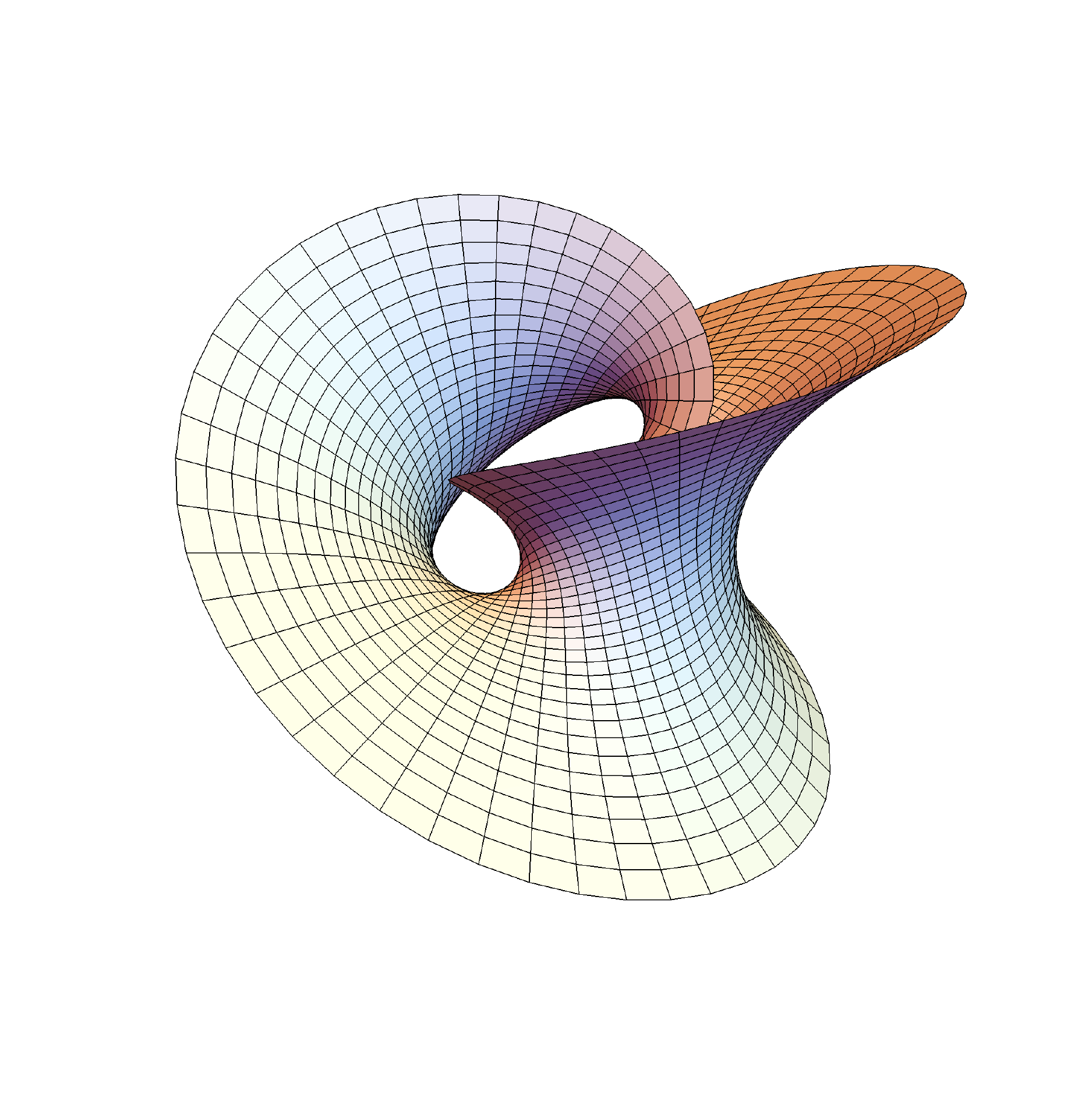}
		\end{minipage}
		\\
		
		\begin{minipage}{0.5\hsize}
			\hspace{30pt}
			\vspace{10pt}
			\includegraphics[width=1.6cm, angle=270]{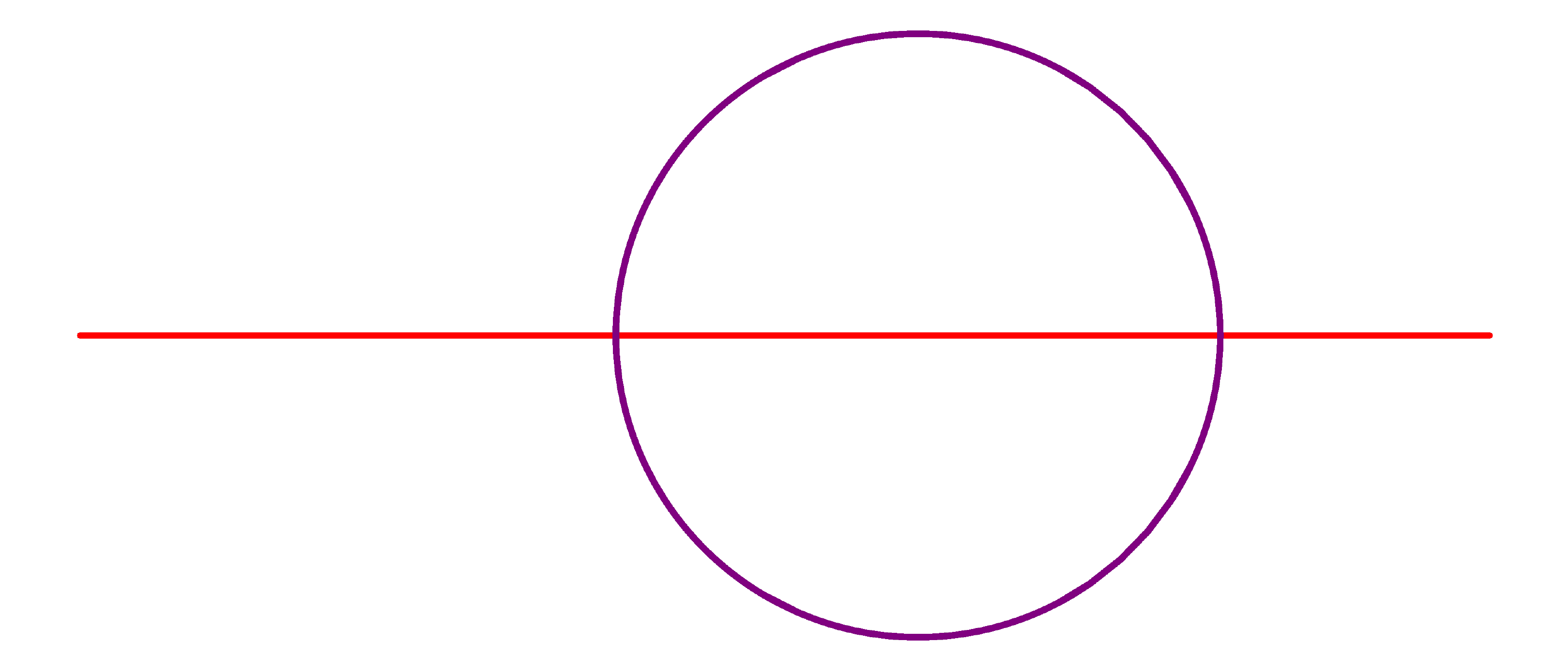}
		\end{minipage}
		
		\begin{minipage}{0.5\hsize}
			\centering
			\includegraphics[width=2.5cm, angle=270]{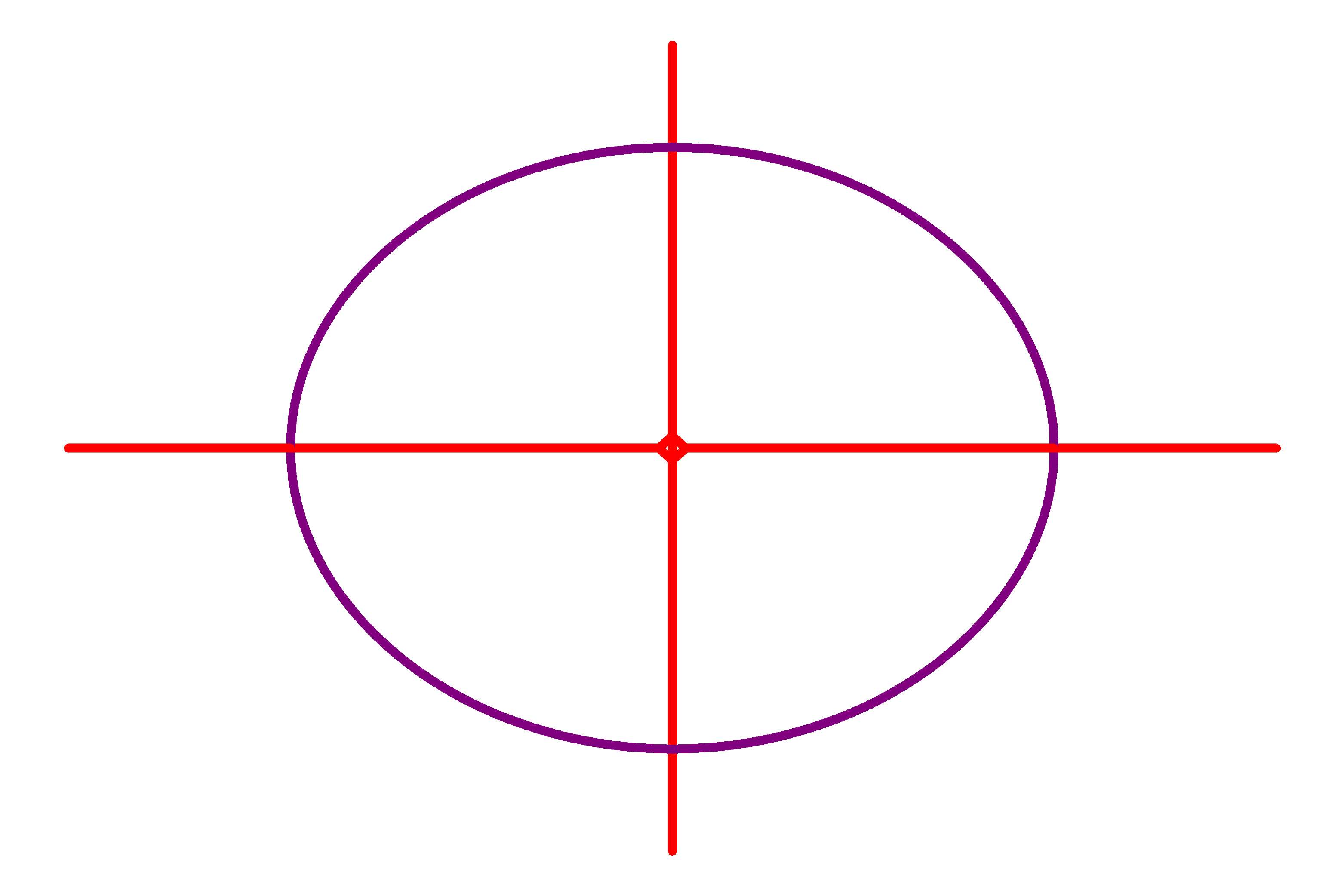}
		\end{minipage}
		
		\\
		\begin{minipage}{0.45\hsize}
			\centering
			minimal M\"{o}bius strip
		\end{minipage}
		
		\begin{minipage}{0.45\hsize}
			\centering
			minimal once-punctured Klein bottle
		\end{minipage}	
	\end{tabular}
	\caption{Nonorientable complete minimal surfaces (top) and their intersection with the $(x_1,x_2)$-plane (bottom).}\label{fig:nonorimin}
\end{figure}
	
L\'{o}pez's example was generalized to complete nonorientable minimal surfaces with arbitrary genus \cite{LM}(See Figure \ref{fig:high_genus_nonorimin}). 
Moreover, many nonorientable minimal surfaces were found in \cite{Ku}, \cite{MW}, \cite{O} and so on.

\begin{figure}[htbp]
	\centering
	\begin{tabular}{c}
		\begin{minipage}{0.31\hsize}
			\centering
			\includegraphics[width=3.6cm]{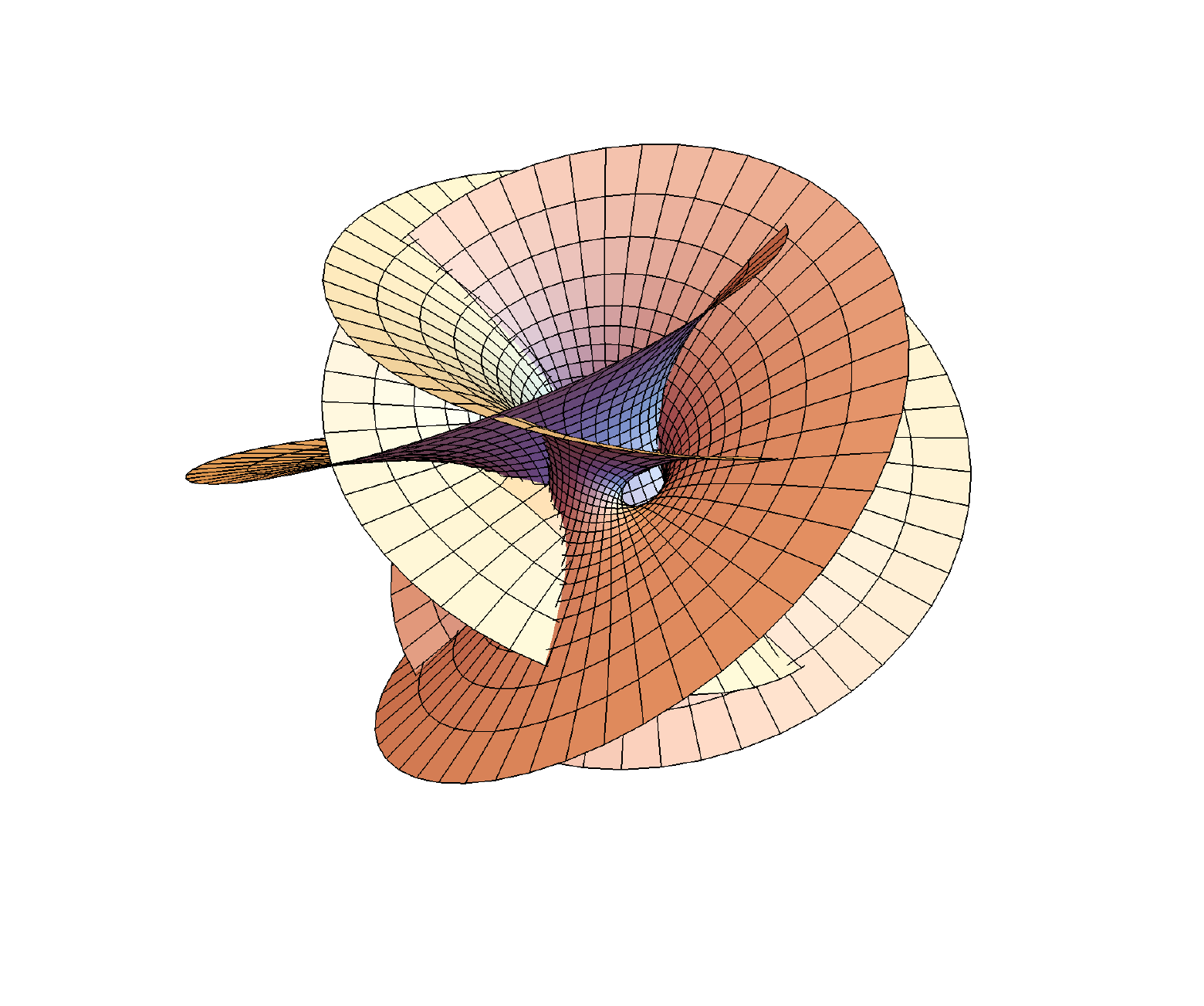}
		\end{minipage}
	
		\begin{minipage}{0.31\hsize}
			\centering
			\includegraphics[width=3.6cm]{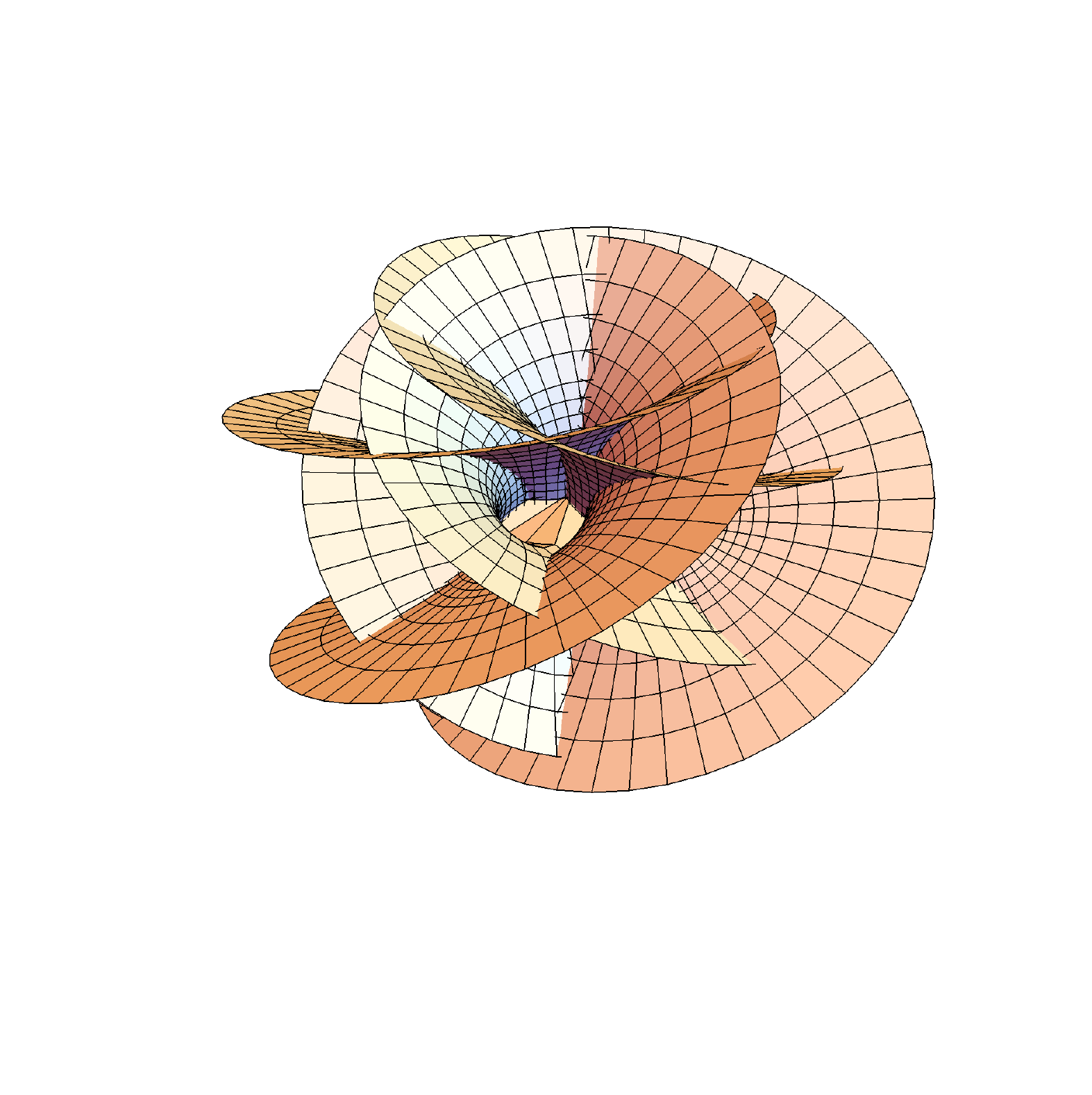}
		\end{minipage}
		
		\begin{minipage}{0.31\hsize}
			\centering
			\includegraphics[width=3.6cm]{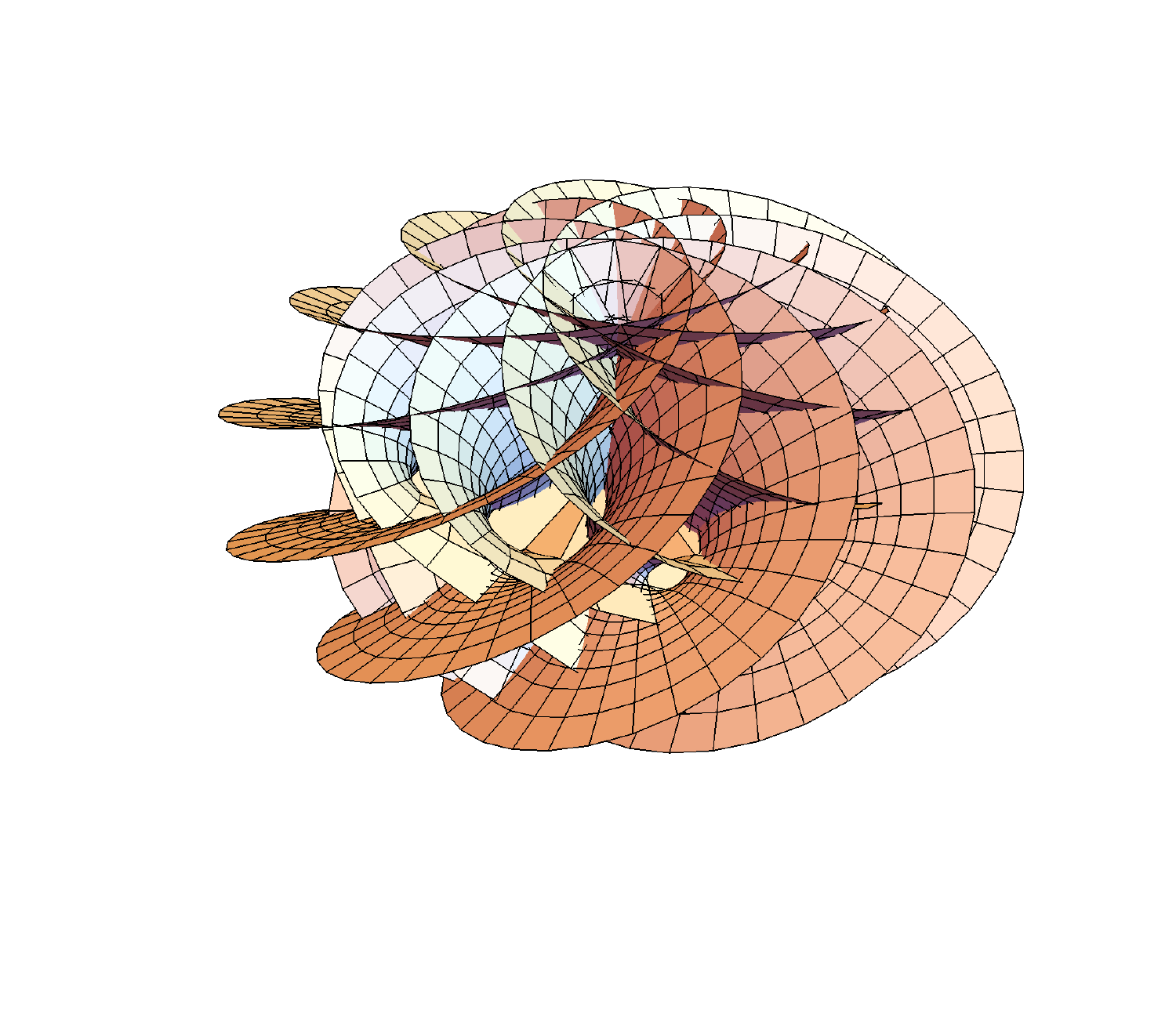}
		\end{minipage}
		\\
		\begin{minipage}{0.31\hsize}
			\centering
			\includegraphics[width=2.8cm]{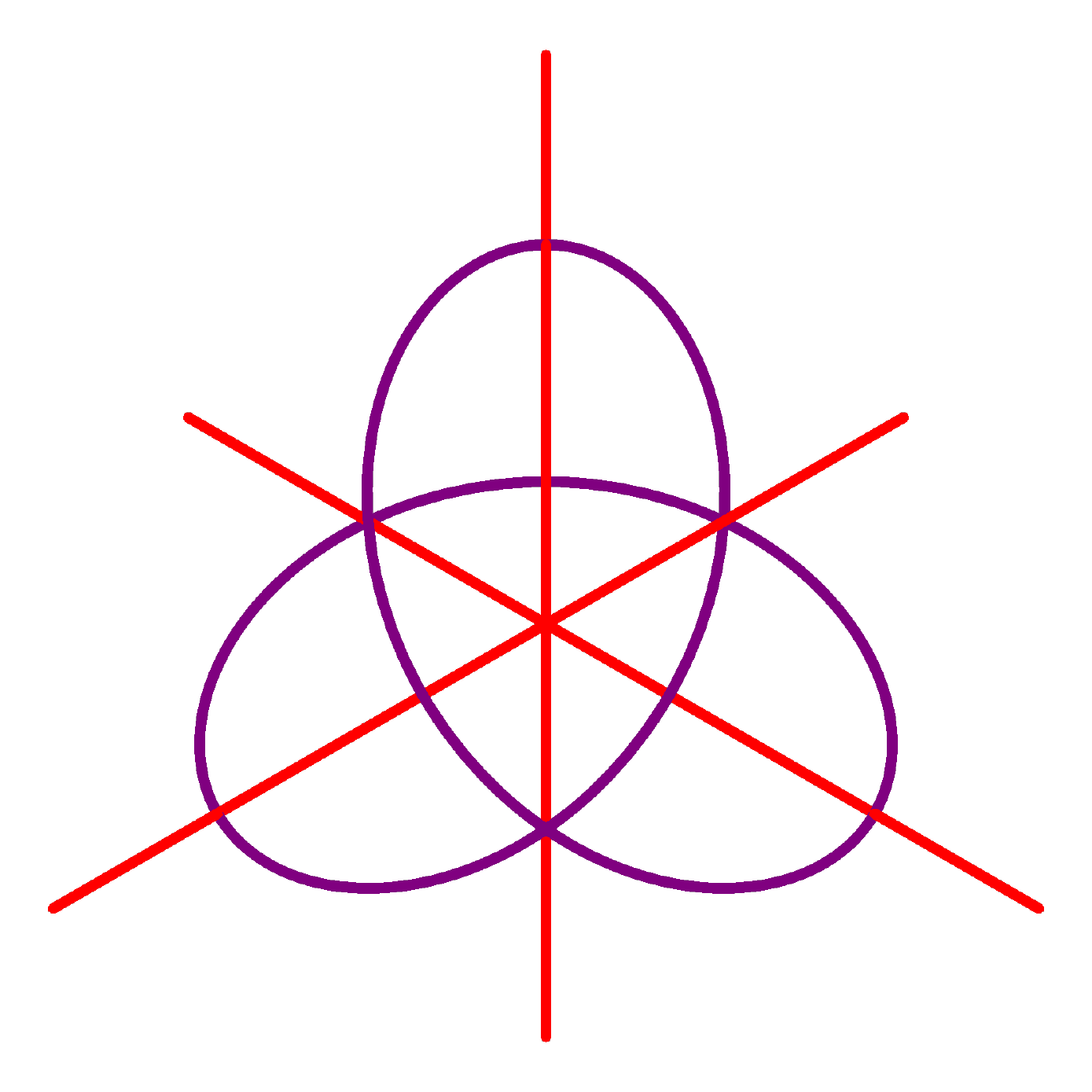}
		\end{minipage}
	
		\begin{minipage}{0.31\hsize}
			\centering
			\includegraphics[width=2.8cm]{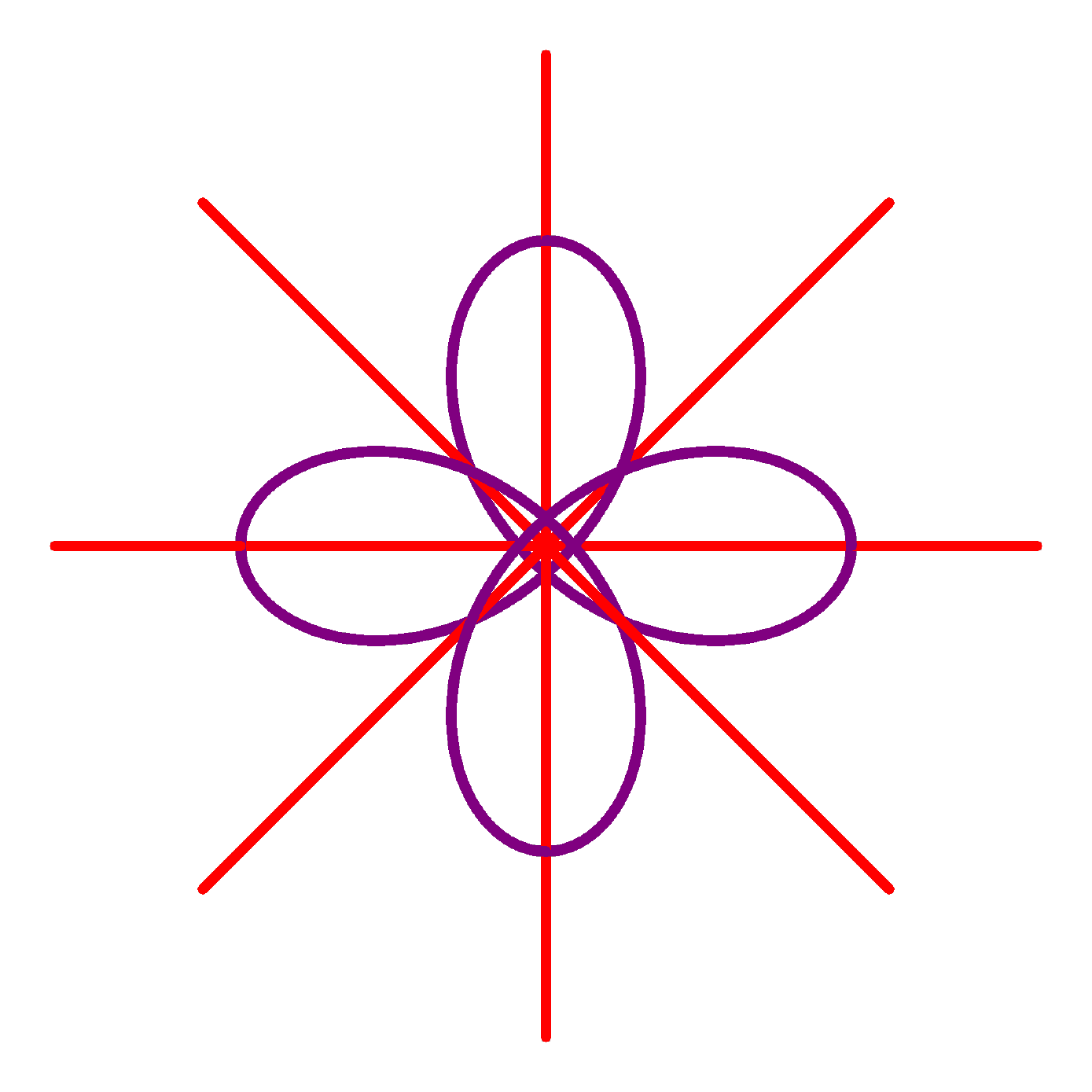}
		\end{minipage}
		
		\begin{minipage}{0.31\hsize}
			\centering
			\includegraphics[width=2.6cm]{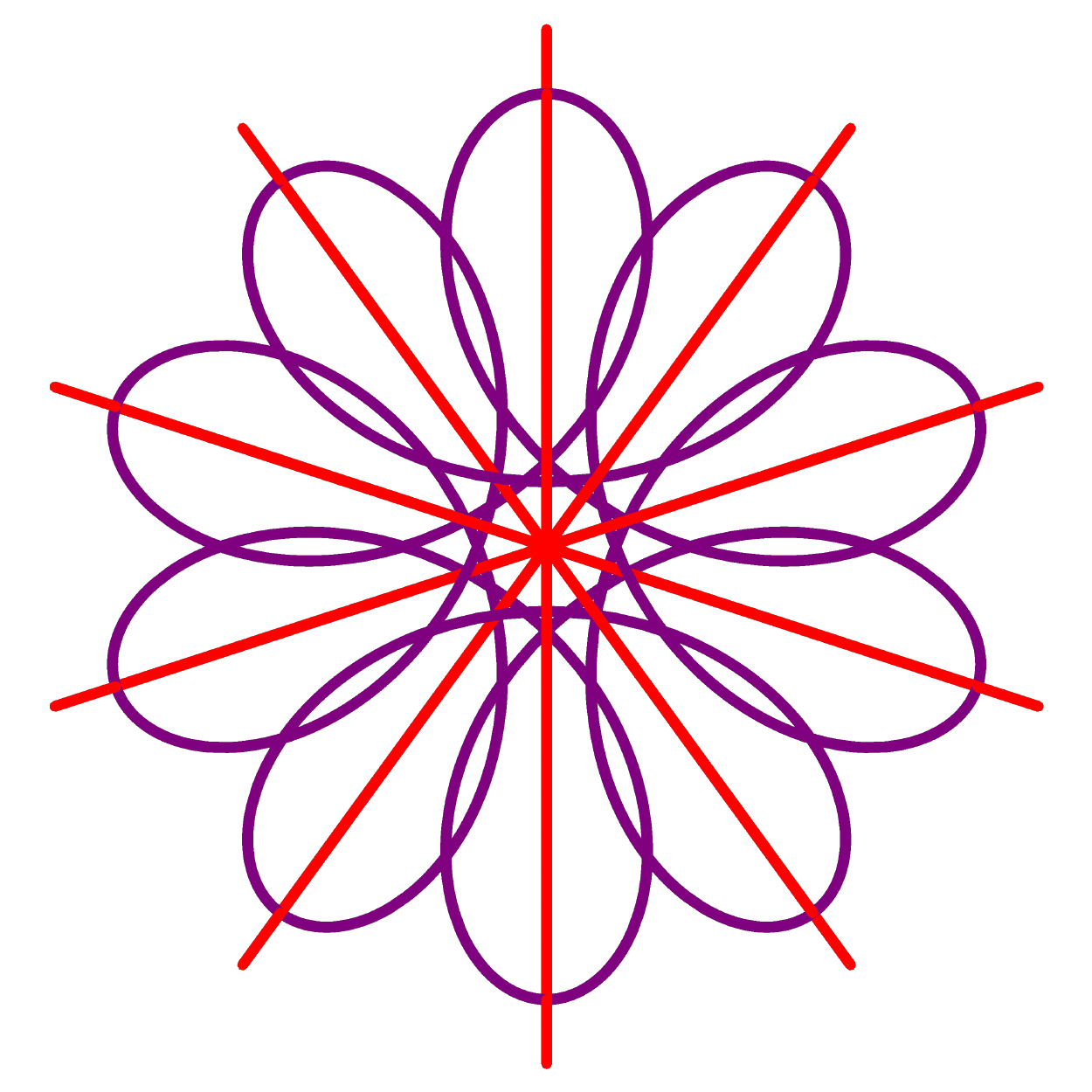}
		\end{minipage}
	\end{tabular}
\caption{Higher genus nonorientable minimal surfaces (top) and their intersection with the $(x_1,x_2)$-plane (bottom).}\label{fig:high_genus_nonorimin}
\end{figure}

A maximal surface in Lorentz-Minkowski 3-space $\L^3$ is a spacelike surface with zero mean curvature. 
Like in the case of minimal surfaces, there is a Weierstrass-type representation formula \cite{Ko} by using holomorphic data.
So, maximal surfaces share many local properties with minimal surfaces.
However, their global behavior is quite different.
In fact, the only complete maximal surface in $\L^3$ (without singularities) is the plane \cite{C, CY}.
Maximal surfaces with singularities are investigated in \cite{ER}, \cite{FeL}, \cite{FLS1}, \cite{FLS2}, \cite{FSUY}, \cite{UY} and so on. 
M. Kokubu and M. Umehara investigated  co-orientability of maximal surfaces in \cite{KU}.

The first author and F. J. L\'{o}pez studied some basic aspects of the global theory of nonorientable maximal surfaces with singularities in \cite{FuL}.
Among other things, they exhibited some fundamental examples (See Figure \ref{fig:nonorimax}) and obtained some classification results; the construction of higher genus examples was postponed.

\begin{figure}[htbp]
	\centering
	\begin{minipage}{0.24\hsize}
 		\centering
 		\includegraphics[width=2.9cm]{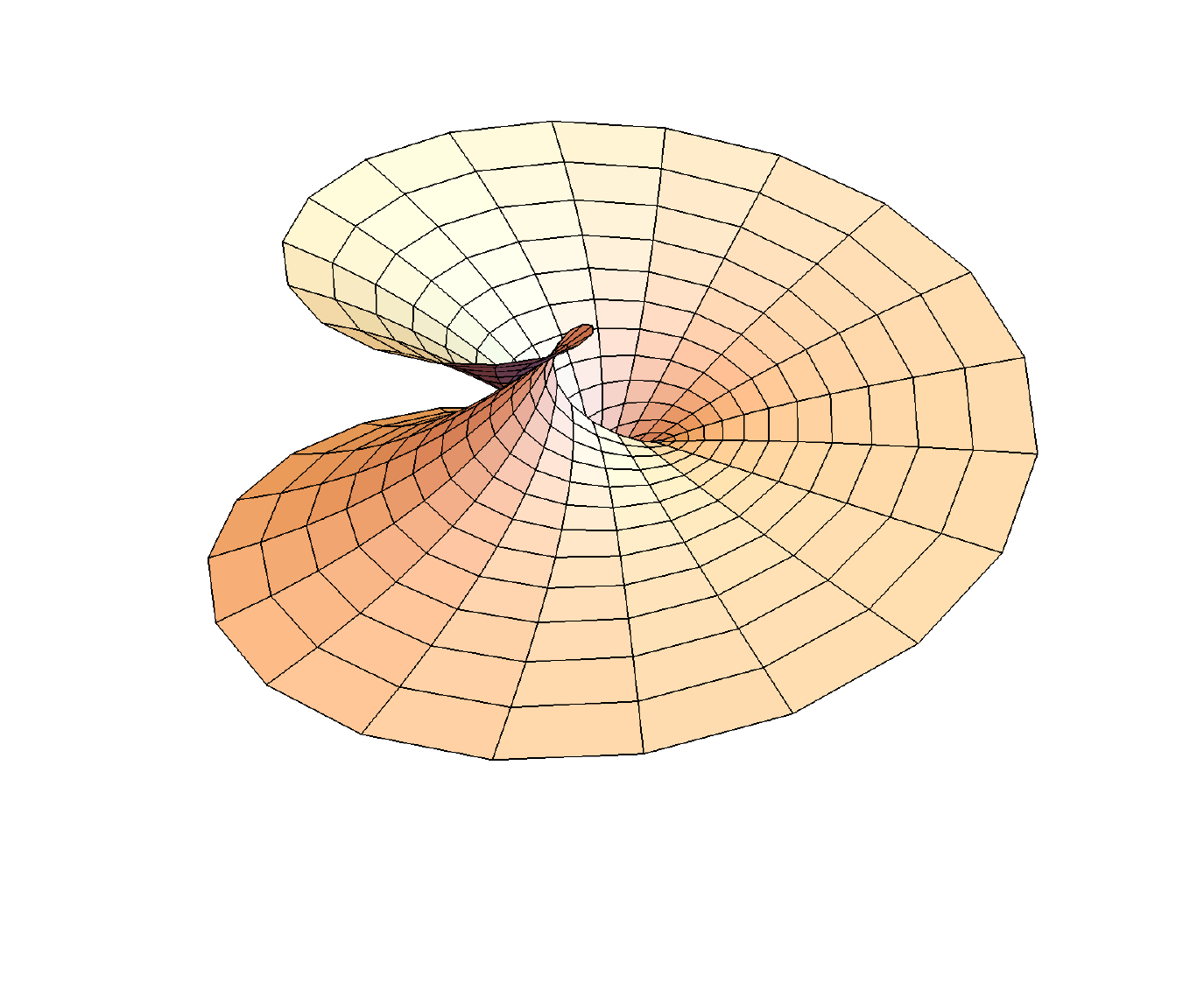}
	\end{minipage}
	\begin{minipage}{0.23\hsize}
		\centering
		\includegraphics[width=2.6cm]{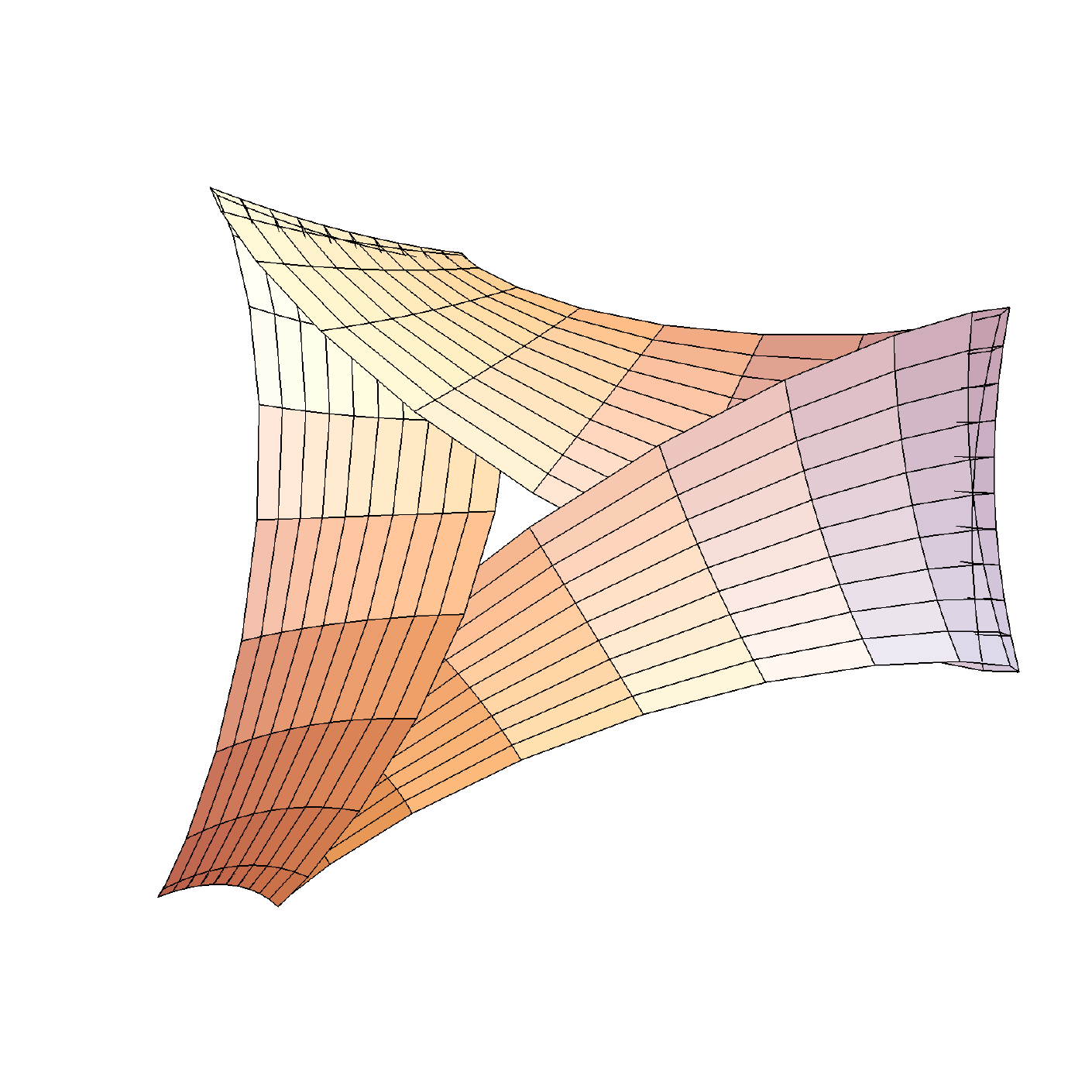}
	\end{minipage}
	\begin{minipage}{0.23\hsize}
 		\centering
 		\includegraphics[width=2.5cm]{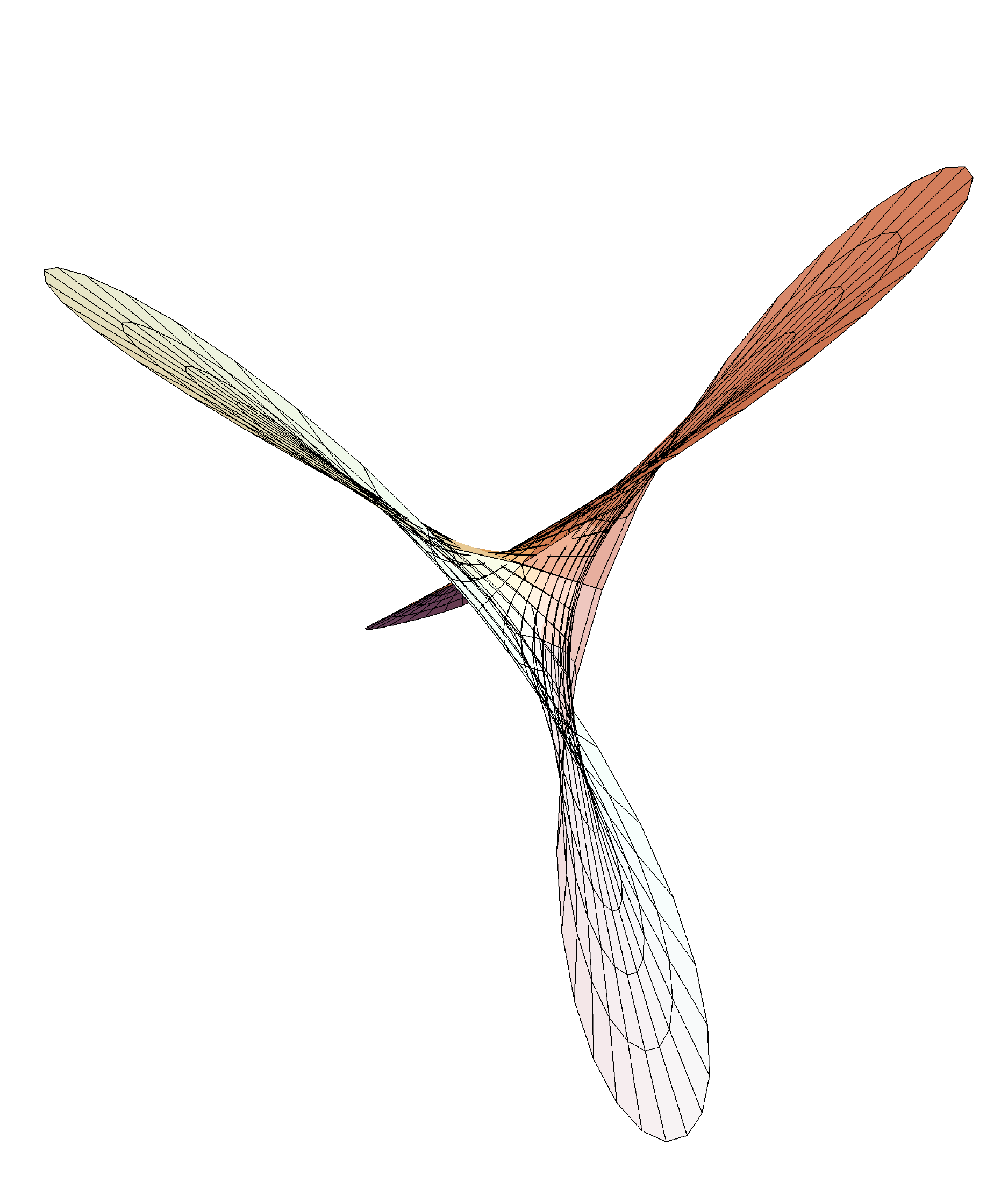}
	\end{minipage}
	\begin{minipage}{0.23\hsize}
		\centering
		\includegraphics[width=2.5cm]{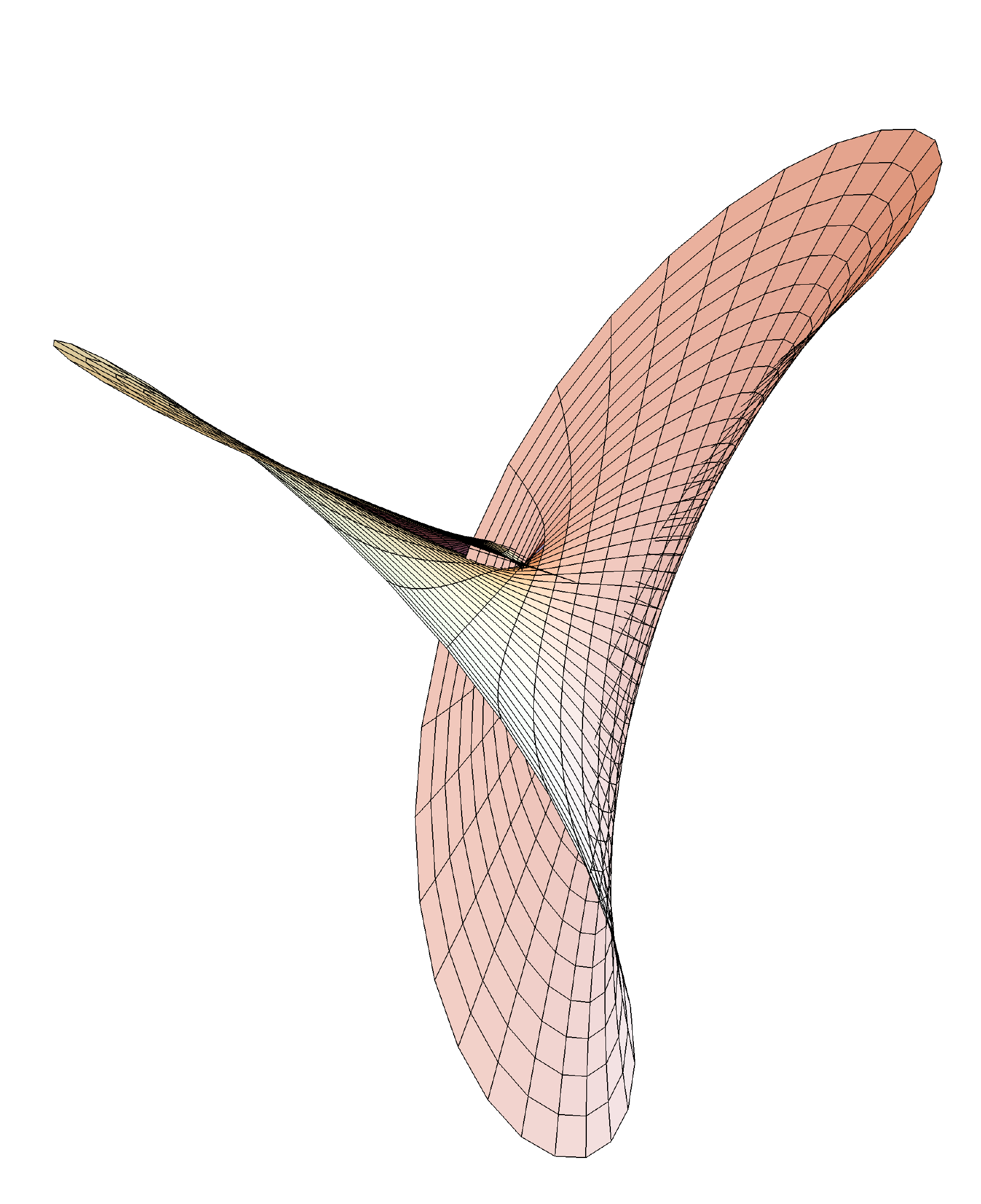}
	\end{minipage}
	\\
	\begin{minipage}{0.24\hsize}
		\hspace*{10pt}
		\includegraphics[width=2.3cm]{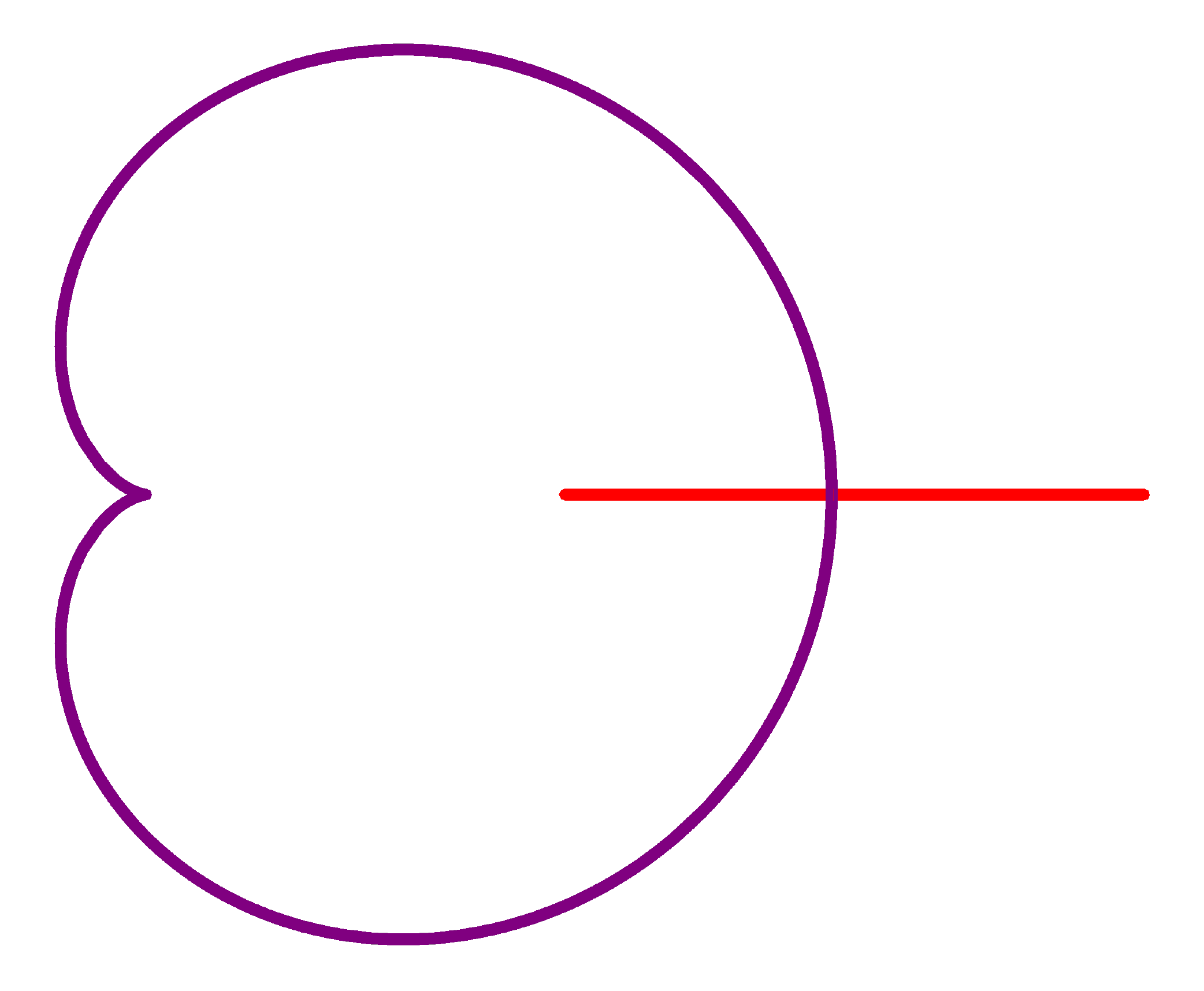}
	\end{minipage}
	\begin{minipage}{0.23\hsize}
		\centering
		\includegraphics[width=2.2cm]{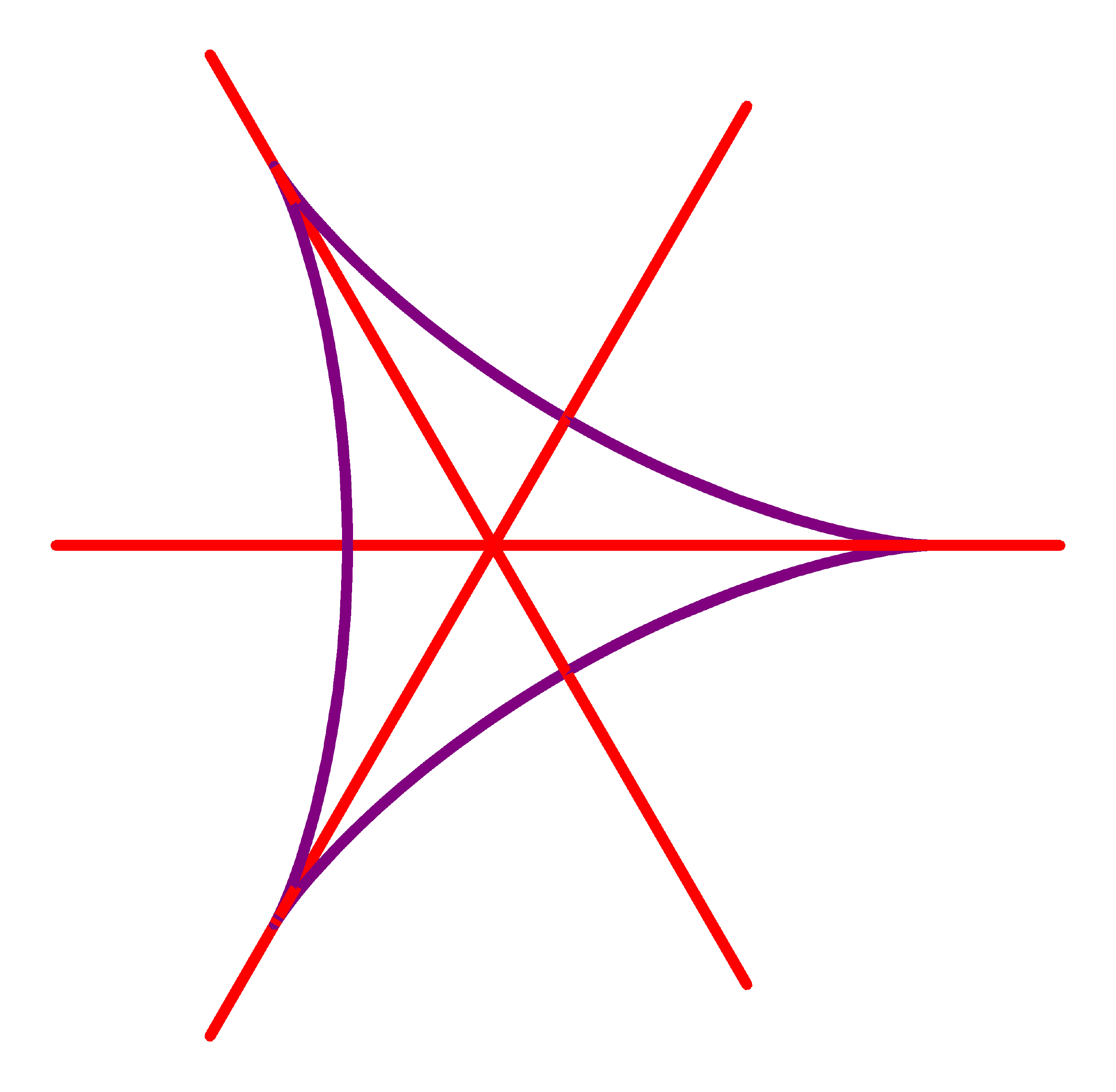}
	\end{minipage}
	\begin{minipage}{0.23\hsize}
 		\hspace*{10pt}
 		\includegraphics[width=2.2cm]{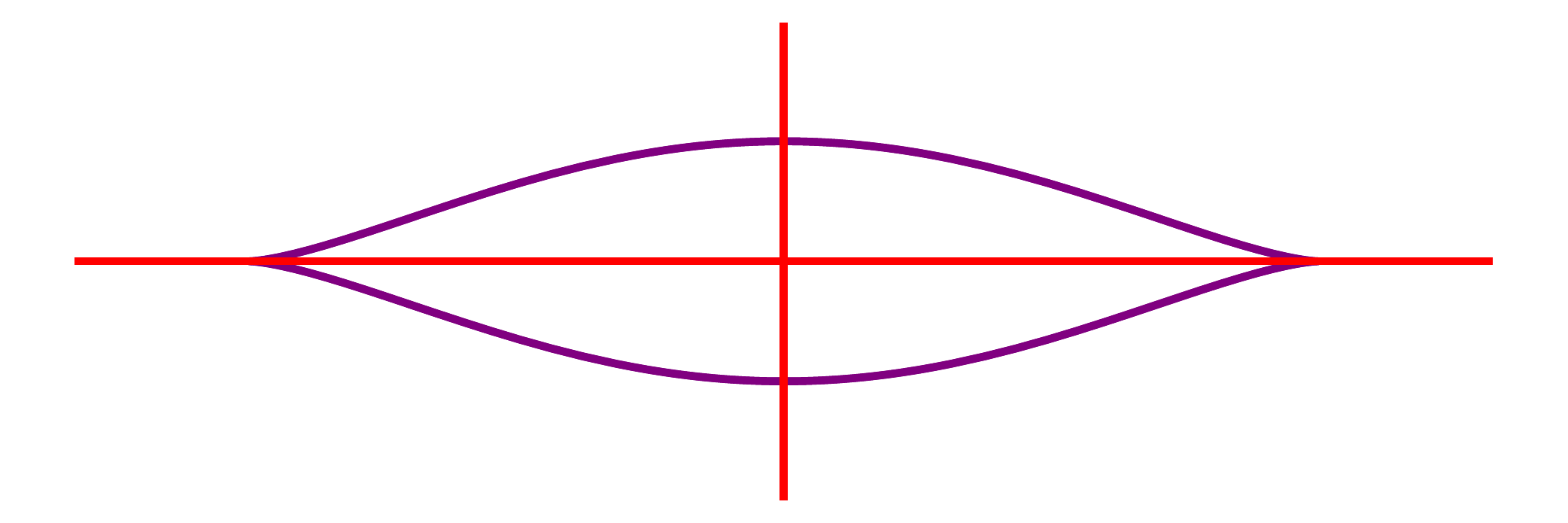}
	\end{minipage}
	\begin{minipage}{0.23\hsize}
		\centering
		\includegraphics[width=2.2cm]{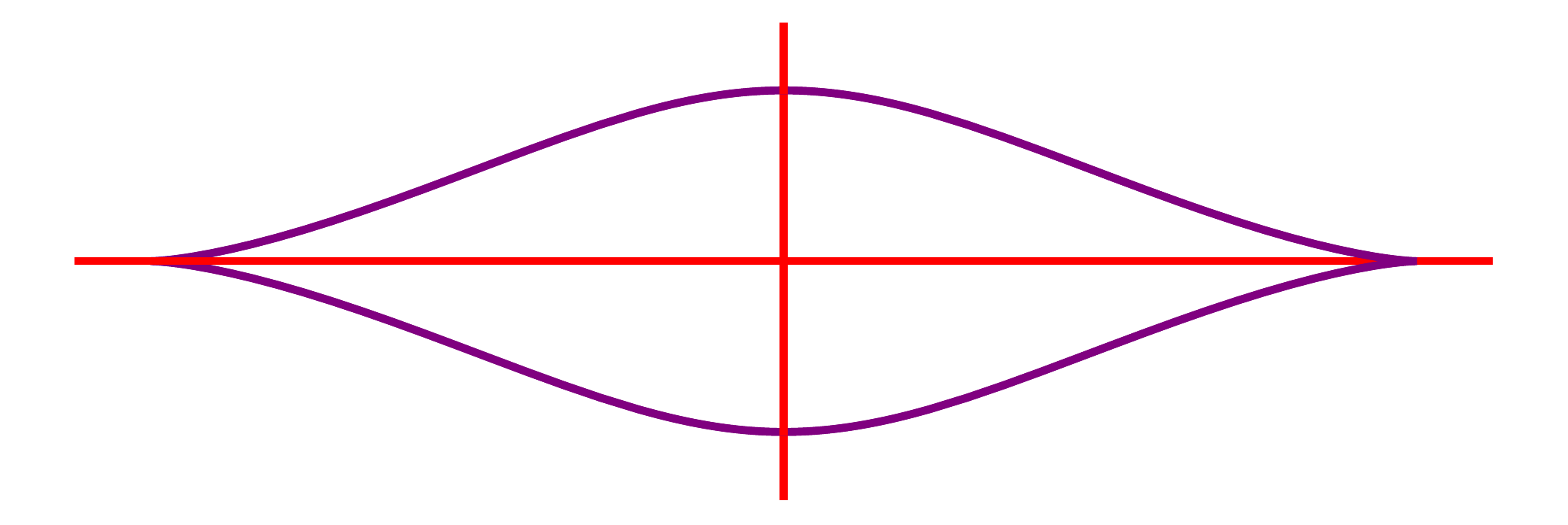}
	\end{minipage}
\\
	\begin{minipage}{0.47\hsize}
 		\centering
 		maximal M\"{o}bius strips
	\end{minipage}
	\begin{minipage}{0.47\hsize}
		\centering
		maximal once-punctured Klein bottles
	\end{minipage}
 \caption{Nonorientable maximal surfaces (top) and their intersection with the $(x_1,x_2)$-plane (bottom)}\label{fig:nonorimax}
 \end{figure} 

In this paper we construct nonorientable maximal surfaces of high genus based on \cite{LM}.
The main difficulty is to solve the period problem associated to the Weierstrass data, that now is posed on a Riemann surface with higher homology group.

This paper is organized as follows:
In Section \ref{sec:pre}, we recall the definition of nonorientable maximal surface with singularities and introduce some basic properties of these surfaces.
Section \ref{sec:exis} is devoted to the construction of higher genus nonorientable maximal surfaces.
In Section \ref{sec:sym}, we discuss the symmetry groups of the surfaces.

\section{Preliminaries}\label{sec:pre}
Let $\L^3$ be the three dimensional Lorentz-Minkowski space with the metric $\ii{~}{~} := dx_1^2 + dx_2^2-dx_3^2$. An immersion $f : M \to \L^3$ of a 2-manifold $M$ into $\L^3$ is called  {\it{spacelike}} if  the induced metric $ds^2 = f^*\ii{~}{~}$ is positive definite. A spacelike immersion $f : M \to \L^3$ is said to be {\it{maximal}} if the mean curvature vanishes identically.

\begin{Def}[Maxfaces \cite{UY}, see also \cite{FKKRSUYY}]
	Let $M$ be a Riemann surface and  $f : M \to \L^3$ a $C^{\infty}$-map. 
	$f$ is called a maxface if there exists an open dense subset $W$ of $M$ so that the restriction $f|_{W} : W \to \L^3$ is a conformal maximal immersion and the rank of $df$ at $p$ is positive for all $p \in M$.
	A point $p \in M$ is said to be singular if rank$(df) <2$ at $p$.
\end{Def}

\begin{Thm}[Weierstrass-type representation \cite{Ko, UY}]\label{th:wr}
	Let $(g, \eta)$ be a pair consisting of a meromorphic function $g$ and a holomorphic differential $\eta$ on a Riemann surface $M$ such that 
	\begin{equation}\label{eq:liftmetric}
		(1+|g|^2)^2\eta\bar{\eta}
	\end{equation}
	gives a Riemannian metric on $M$. 
	Set
	\begin{equation}\label{eq:wr}
		\Phi =
		 \begin{pmatrix}
			\phi_1 \\
			\phi_2 \\
			\phi_3
		\end{pmatrix}
		=
		\begin{pmatrix}
			(1 + g^2)\eta \\
			i(1 - g^2)\eta \\
			2g\eta
		\end{pmatrix}
		.
	\end{equation}
	Suppose that
	\begin{equation}\label{eq:period}
		\Re \oint_{\gamma} \Phi = {\bf{0}}
	\end{equation}
	hols for any $\gamma \in H_1(M, \Z).$ Then
	\begin{equation}
		f = \Re \int_{z_0}^{z} \Phi : M \to \L^3, \qquad (z_0 \in M)
	\end{equation}
	is a maxface. 
	The singular set of f corresponds to $\{p \in M \: |\: |g(p)|=1\}$, and the first and second fundamental forms of $f$ are given by
	\[
		(1-|g|^2)^2\eta\bar{\eta}\quad and \quad \eta dg + \bar{\eta} d\bar{g}
	\]
	respectively. We call $(M, g, \eta)$ $($or $(g, \eta)$$)$ the Weierstrass data of $f : M \to \L^3$ $($We define the Weierstrass data of a nonorientable maxface later$)$.
\end{Thm}

\begin{Rem}
	The condition (\ref{eq:period}) is called the period problem, which guaranties the well-defindness of $f$. 
	It is equivalent to 
	\begin{equation}\label{eq:period01}
		\oint_{\gamma} g^2\eta + \overline{\oint_{\gamma}\eta} = 0
	\end{equation}
	and
	\begin{equation}\label{eq:period02}
		\Re \oint_{\gamma} g\eta = 0
	\end{equation}
	for any $\gamma \in H_1(M, \Z)$. 
\end{Rem}

\begin{Def}[{\cite[Definition 1.3]{UY}}]
	A maxface $f : M \to \L^3$ is said to be {\it{complete}} if there exists a compact set $C$ and a symmetric $(0, 2)$-tensor $T$ on $M$ such that $T$ vanishes on $M\bash C$ and $ds^2 + T$ is a complete Riemannian metric. 
\end{Def}

\begin{Prop}[{\cite[Proposition 4.5]{UY}}]
	Let $f : M \to \L^3$ be a complete maxface with the Weierstrass data $(M, g, \eta)$. 
	Then the Riemann surface $M$ is biholomorphic to a compact Riemann surface $\overline{M}$ minus a finite number of points $\{p_1,\dots, p_n\}$. 
	Moreover, $g$ and $\eta$ extend meromorphically to $\overline{M}$.
\end{Prop}

We say that a complete maxface $f : \overline{M}\bash \{p_1, \dots, p_n\} \to \L^3$ is of genus $k$ if $\overline{M}$ is a compact Riemann surface of genus $k$.

Let $M'$ be a nonorientable surface with conformal coordinates. 
We denote $\pi : M \to M'$ the orientable conformal double cover of $M'$.
\begin{Def}[{\cite[Definition 2.1]{FuL}}]\label{def:nonori}
	A conformal map $f' : M' \to \L^3$ is said to be a {\it{nonorientable maxface}} if the composition
	\[
		f = f' \circ \pi : M \to \L^3
	\]
	is a maxface. In addition, $f'$ is said to be complete if $f$ is complete. 
\end{Def}
We say that a complete nonorientable maxface $f' : M' \to \L^3$ is of genus $k$ $(k\geq1)$ if the double cover $M$ of $M'$ is biholomorphic to $\overline{M}\bash\{p_1,\dots,p_n\}$, where $\overline{M}$ is a compact Riemann surface of genus $k-1$.

Let $f' : M' \to \L^3$ be a nonorientable maxface, and let $I : M \to M$ denote the antiholomorphic order two deck transformation associated with the orientable double cover $\pi : M \to M'$. Since $f \circ I = f$, we have
\begin{equation}\label{eq:nonorientable}
	g \circ I = \frac{1}{\bar{g}}  \quad {\it{and}} \quad I^{*}\eta = \overline{g^2\eta}.
\end{equation}

Conversely, if $(g, \eta)$ is the Weierstrass data of an orientable maxface $f : M \to \L^3$ and $I$ is an antiholomorphic involution without fixed points in $M$ satisfying (\ref{eq:nonorientable}), then the unique map $f' : M' = M/\langle I \rangle \to \L^3$ satisfying that $f = f' \circ \pi$ is a nonorientable maxface. We call $(M, I, g ,\eta)$  the Weierstrass data of $f' : M' \to \L^3$.

\section{Existence}\label{sec:exis}
Let $\overline{M_k}$ be the compact Riemann surface of genus $k$ defined by 
\[
	\overline{M_k} = \left\{(z, w)\in \Rs^2 \; \left|\; w^{k+1}=\dfrac{z(z-r)}{rz+1}\right.\right\}, \qquad r \in \R\bash\{0, 1\},
\]
where $\Rs = \C \cup \{\infty\}$, and set $M_k = \overline{M_k}\bash \{(0,0),(\infty,\infty)\}$. 
Define
\begin{align}
	\begin{aligned}\label{aligned:wd}
		I : \overline{M_k} \to \overline{M_k}, \qquad I(z, w) = \left(-\dfrac{1}{\bar{z}}, \dfrac{1}{\bar{w}}\right), \\
		g= \frac{w^k(z+1)}{z(z-1)}, \qquad \eta = i \dfrac{(z-1)^2}{z w^k}dz.
	\end{aligned}
\end{align}
Note that $I$ has no fixed points, and $g$ and $\eta$ satisfy (\ref{eq:nonorientable}) and  (\ref{eq:liftmetric}) gives a Riemannian metric on $M_k$. We see that
\[
	\deg(g) = 2(k+1).
\]
See Table \ref{tab:data}.
\begin{table}[h]
	\centering
	\begin{tabular}{c|cccccc}
		$(z, w)$&$(-1, *)$&$(-1/r, \infty)$&$(0, 0)$&$(1, *)$&$(r, 0)$&$(\infty, \infty)$ \\ \hline 
		$g$&$0^1 \times (k + 1)$&$\infty^k$&$\infty^1$&$\infty^1 \times (k+1)$&$0^k$&$0^1$\\ 
		$\eta$&&$0^{2k}$&$\infty^{k+1}$&$0^2 \times (k + 1)$&&$\infty^{k+3}$
	\end{tabular}
\caption{Orders of zeros and poles of $g$ and $\eta$}\label{tab:data}
\end{table}

\begin{Rem}
	In the case $k=1$, changing the coordinate $\zeta=1/z$ and $\omega=i w/z$ in $\overline{M_k}$ and rotating $-\pi/2$ about $x_3$-axis, the Riemann surface $\overline{M_1}$ and the antiholomorphic involution $I$ and $g$ and $\eta$ become
	\begin{gather*}
		\overline{M_1}  = \left\{(\zeta, \omega)\in \Rs^2 \; \left|\; \omega^{2}=\dfrac{\zeta(r\zeta-1)}{\zeta + r}\right.\right\}, \qquad I(\zeta, \omega) = \left(-\frac{1}{\bar{\zeta}}, -\frac{1}{\bar{\omega}}\right),\\
		g = \omega \frac{\zeta+1}{\zeta-1}, \qquad \eta = i\frac{(\zeta-1)^2}{\zeta^2\omega}d\zeta.
	\end{gather*}
	They coincide with the Weierstrass data of the maximal Klein bottle with one end constructed in \cite{FuL}.
\end{Rem}

\begin{Thm}[Existence]\label{th:ex}
For each $k \geq 1$, there are exactly two real values $r_1, r_2 \in \R\bash\{0\}$ for which the maxface
	\[
		f = \Re\int_{z_0}^{z} \Phi : M_k \to \L^3
	\]
is well-defined and induces a one-ended complete nonorientable maxface $f' : M_k / \langle I \rangle \to \L^3$ of genus $k + 1$, where $\Phi$ is the $\C^3$-valued holomorphic differential given by $(\ref{eq:wr})$.
See Figure $\ref{fig:hgn}$.
\end{Thm}

\begin{figure}[htbp]
\centering
	\begin{tabular}{c}
		\begin{minipage}{0.23\hsize}
			\centering
			\includegraphics[width=2.8cm]{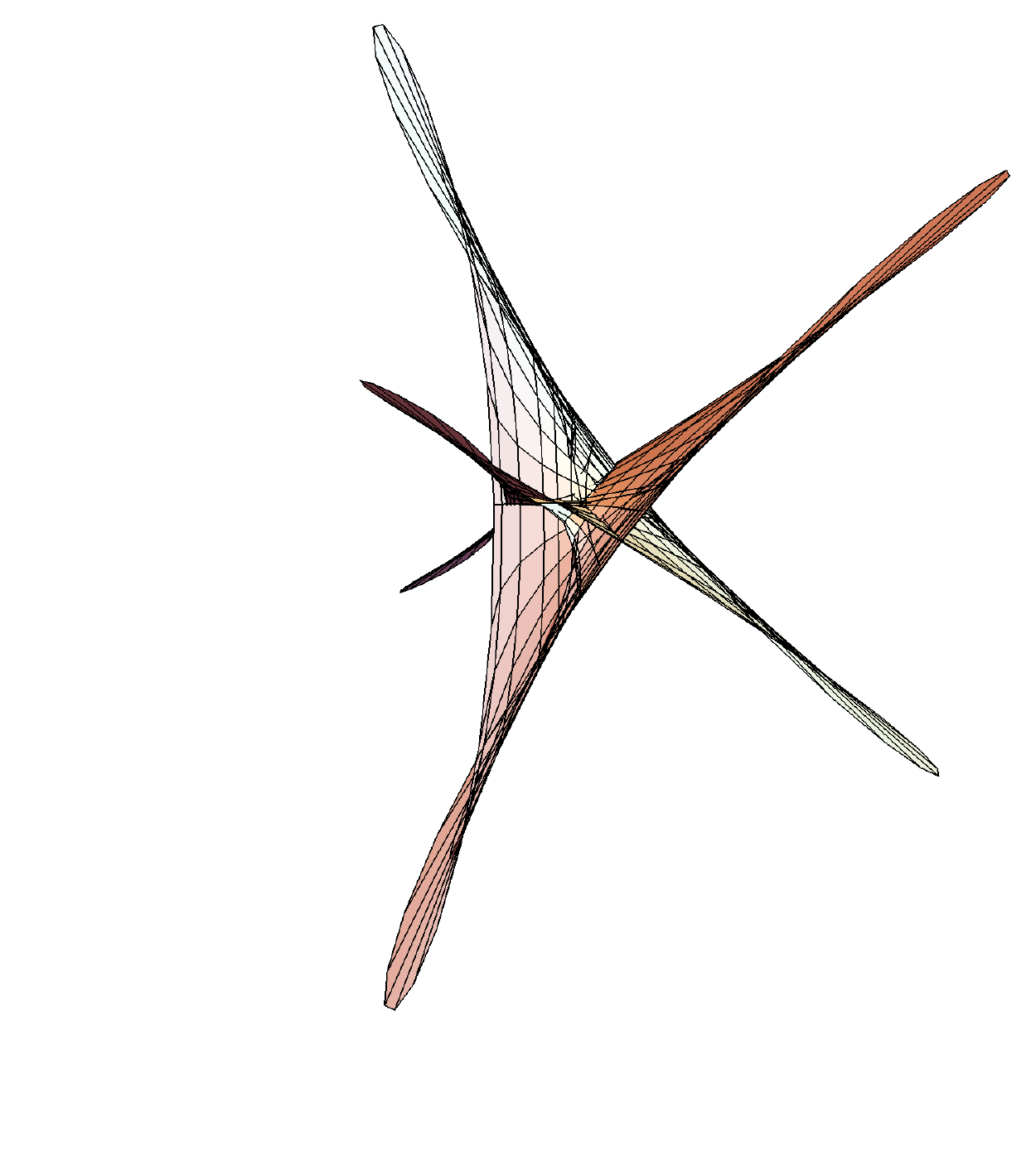}
		\end{minipage}
		
		\begin{minipage}{0.23\hsize}
			\centering
			\includegraphics[width=2.8cm]{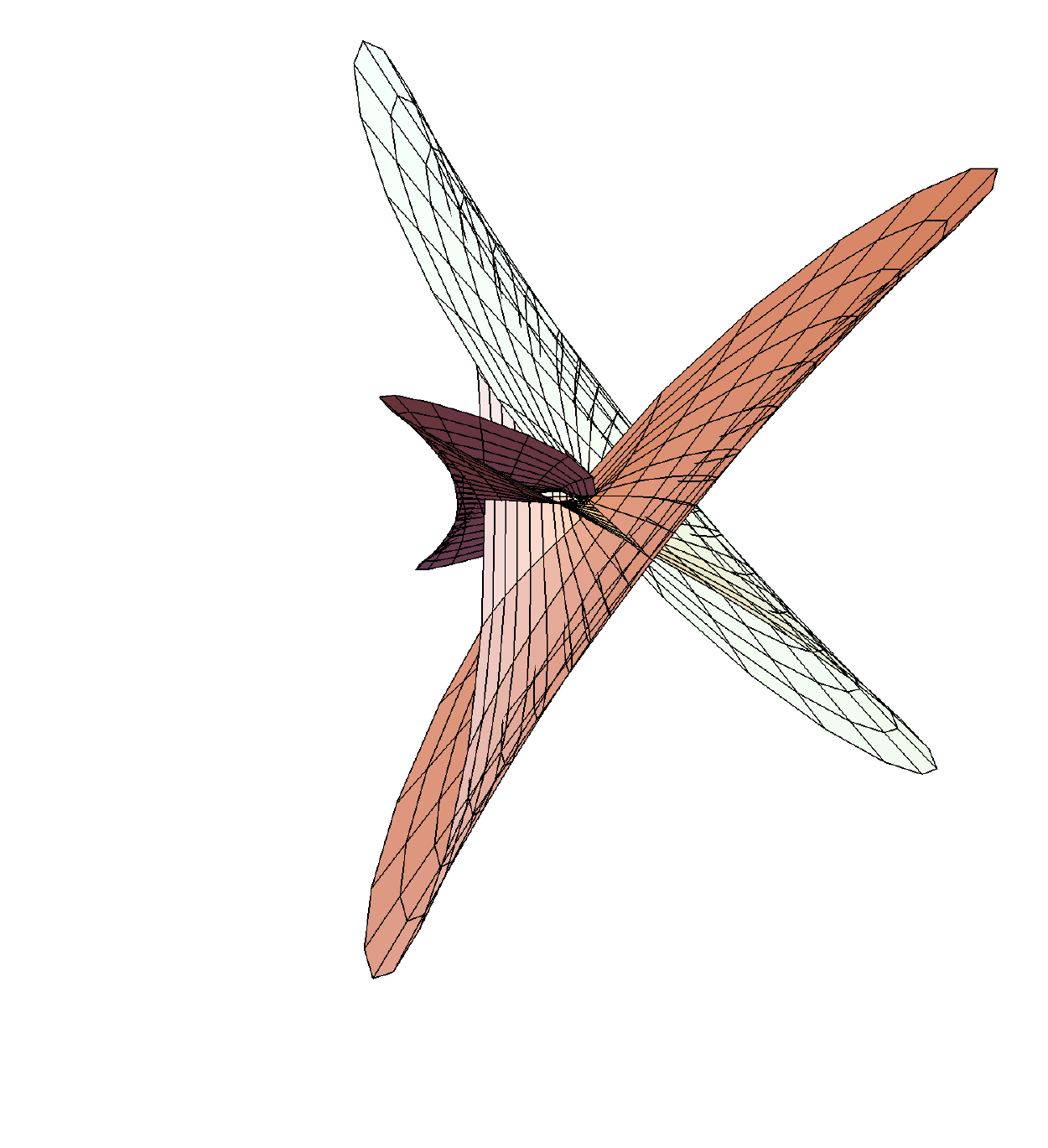}
		\end{minipage}
		
		\begin{minipage}{0.23\hsize}
			\centering
			\includegraphics[width=2.8cm]{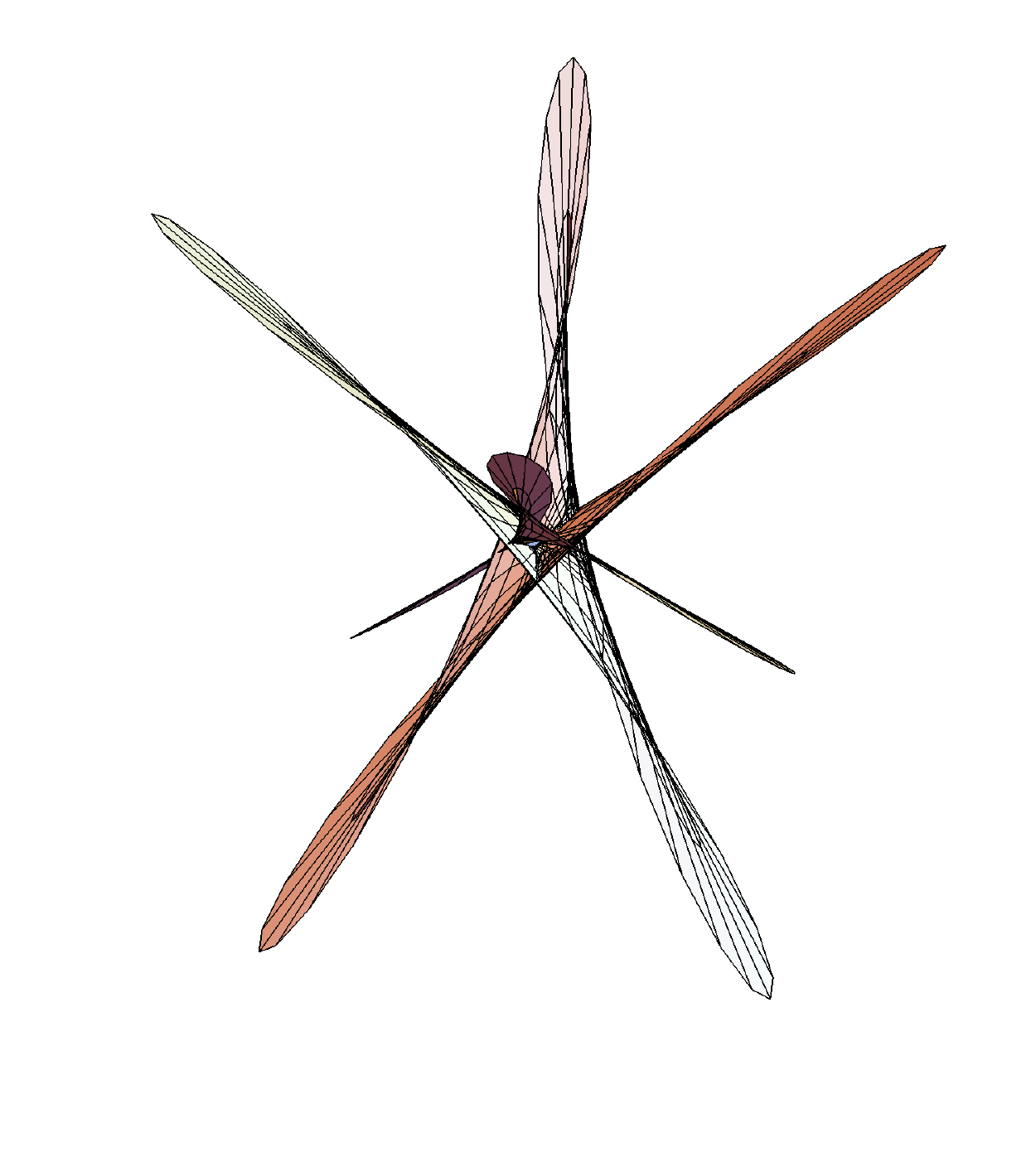}
		\end{minipage}
		
		\begin{minipage}{0.23\hsize}
			\centering
			\includegraphics[width=2.8cm]{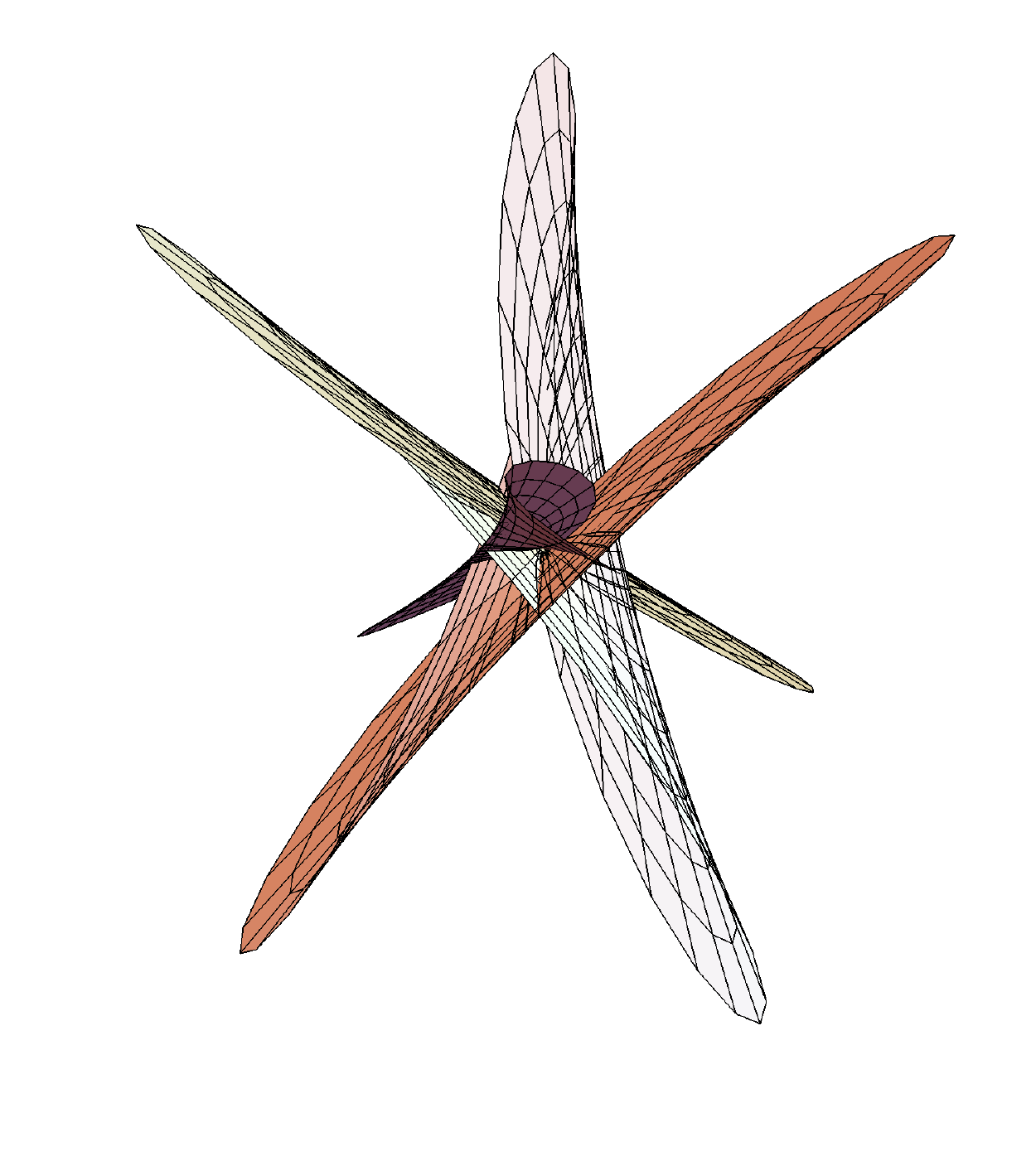}
		\end{minipage}
\\
		\begin{minipage}{0.23\hsize}
			\centering
			\includegraphics[width=2cm]{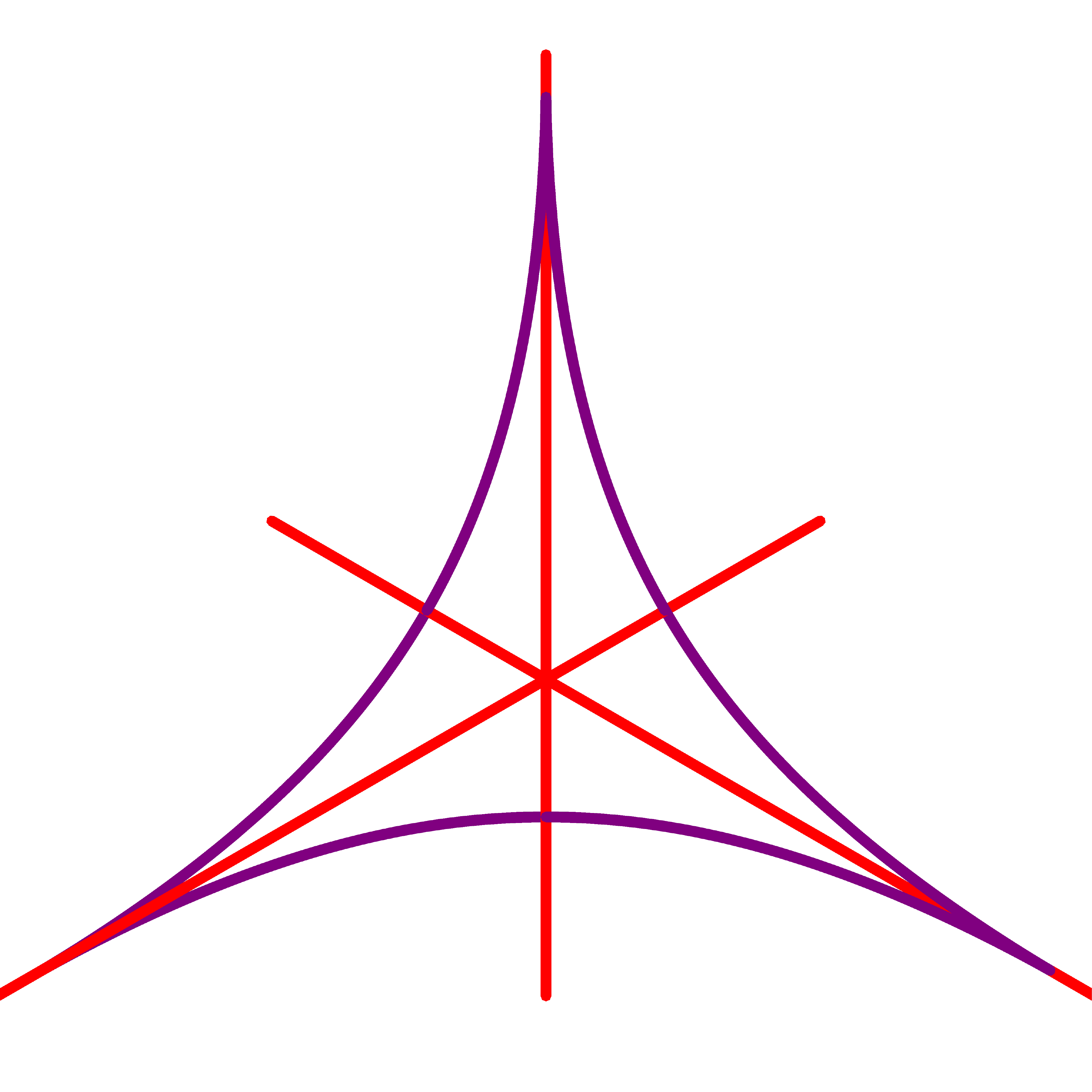}
		\end{minipage}

		\begin{minipage}{0.23\hsize}
			\centering
			\includegraphics[width=2cm]{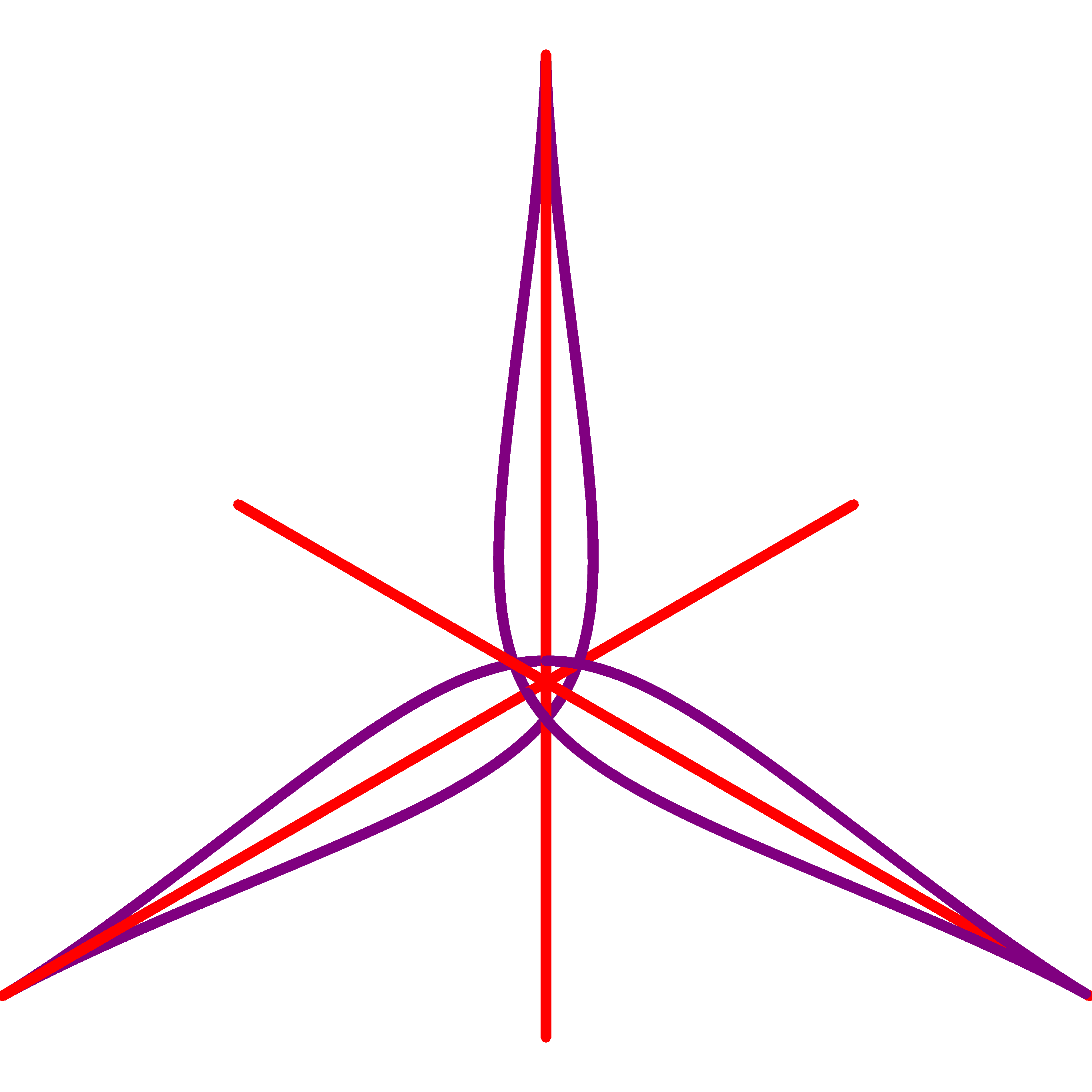}
		\end{minipage}

		\begin{minipage}{0.23\hsize}
			\centering
			\includegraphics[width=2cm]{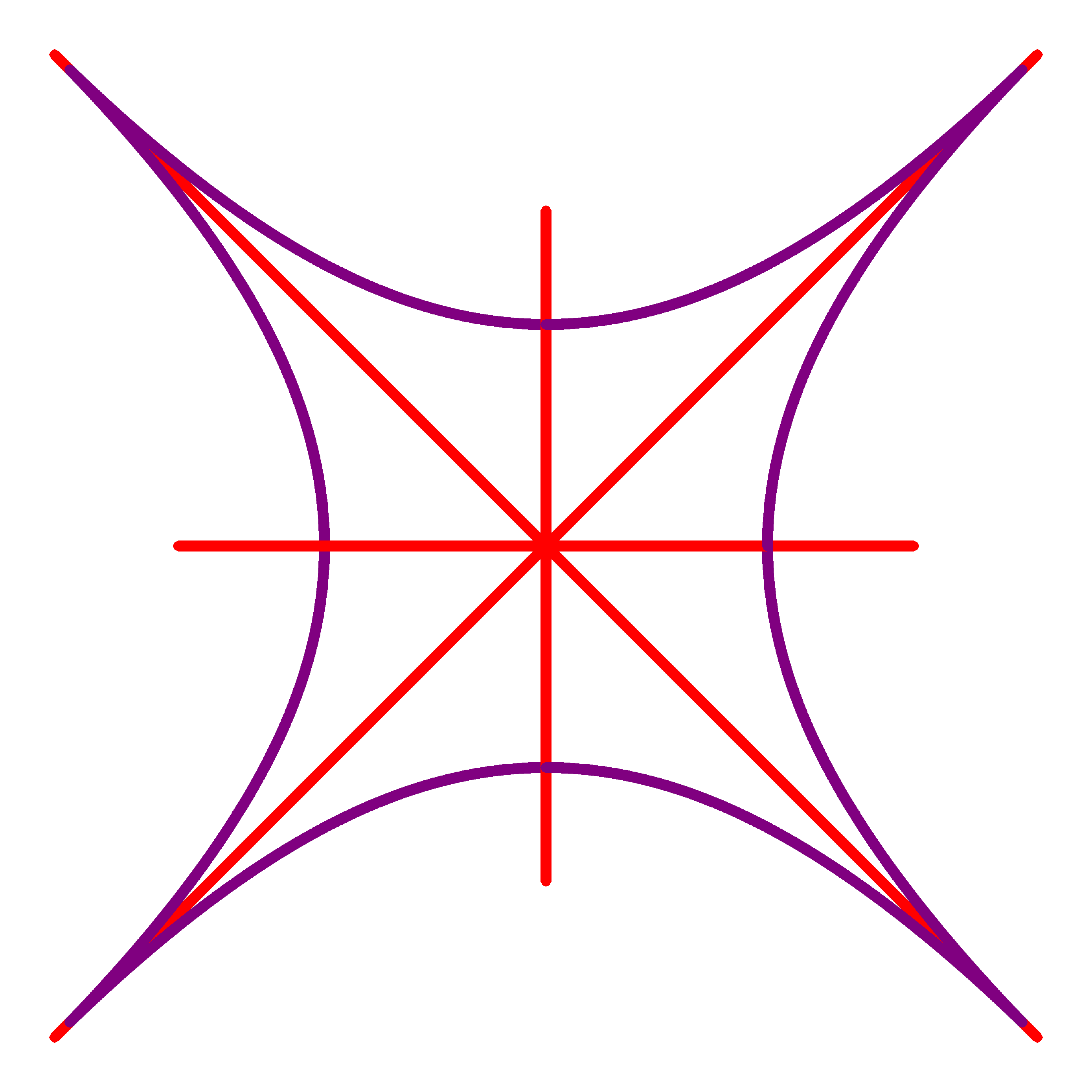}
		\end{minipage}
		
		\begin{minipage}{0.23\hsize}
			\centering
			\includegraphics[width=2cm]{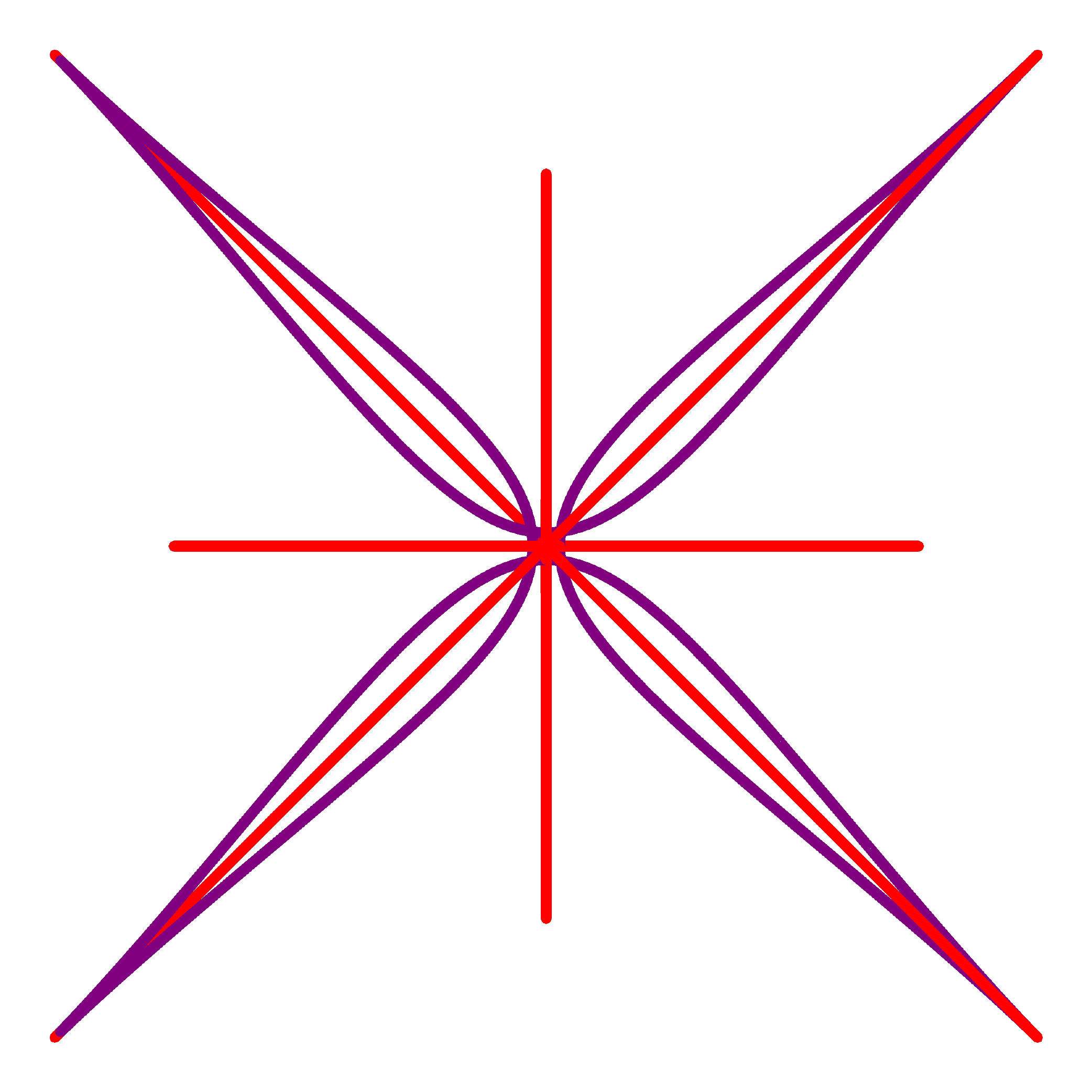}
		\end{minipage}
\\
		\begin{minipage}{0.4\hsize}
			\centering
			$k=2$, Left: $r \approx 0.478169$, Right: $r \approx 0.807158$.
		\end{minipage}
		
		\begin{minipage}{0.4\hsize}
			\centering
			$k=3$, Left: $r \approx 0.615965$, Right: $r \approx 0.859345$.
		\end{minipage}
	\end{tabular}
\caption{Nonorientable maximal surfaces of genus $k+1$ with one end (top) and their intersection with the $(x_1,x_2)$-plane (bottom).}\label{fig:hgn}
\end{figure}

\begin{proof}
Once we solve the period problem for some $r$, the arising immersion is clearly a complete maxface;
just observe that $|g|\neq1$ at the two ends of $M_k$.
In the case $k = 1$, the proof is given in \cite{FuL}. Here we consider the case $k \geq 2$. $f$ is well-defined if and only if $\phi_j$, $j=1, 2, 3$, have no real periods on loops in $M_k$. 
We first observe that 
	\begin{equation}\label{eq:geta}
		\phi_{3} = 2 g \eta = d\left(\frac{2i(z^2 + 1)}{z}\right)
	\end{equation}
is an exact 1-form on $\overline{M_k}$ and (\ref{eq:period02}) is satisfied.
It is easy to check that $\phi_1, \phi_2$ have no residues at the topological ends of $M_k$ (see (\ref{eq:exp})). 
So, it suffices to prove (\ref{eq:period01}) for any homology basis of $\overline{M_k}$.

On the other hand, for any loop $\gamma$ in $\overline{M_k}$,
	\[
		\oint_{\gamma} \phi_j = \oint_{I_{*}(\gamma)} I^{*}(\phi_j) = \oint_{I_{*}(\gamma)}\overline{\phi_j}
	\]
and so
	\[
		2 \Re \oint_{\gamma} \phi_j = \oint_{\gamma + I_{*}(\gamma)} \phi_j.
	\]	

Therefore, $f$ is well-defined on $M_k$ if and only if 
	\begin{equation}\label{eq:nonori_period}
		\oint_{\gamma + I_{*}(\gamma)} \phi_j = 0, \quad \gamma \in H_1(\overline{M_k}, \Z), \quad j=1, 2.
	\end{equation}
	
Define conformal maps $\kappa_j : \overline{M_k} \to \overline{M_k}$ $(j = 1,2)$ as follows:
	\begin{equation}\label{eq:k1k2}
		\kappa_1(z, w) = (z, e^{\frac{2\pi i}{k+1}}w), \qquad \kappa_2(z, w) = (\bar{z}, \bar{w}).
	\end{equation}
Then we have the following:
	\begin{equation}\label{eq:sym}
		\kappa_1^*\Phi = K_1 \Phi, \qquad  
		\kappa_2^*\Phi = K_2\overline{\Phi},
	\end{equation}
where
	\begin{equation}\label{eq:K}
		K_1 =
		\begin{pmatrix}
			\cos \frac{2\pi}{k+1} & \sin \frac{2\pi}{k+1} & 0\\[4pt]
			-\sin \frac{2\pi}{k+1} & \cos \frac{2\pi}{k+1} & 0\\[4pt]
			0&0&1
		\end{pmatrix},
		\qquad  K_2 =
		\begin{pmatrix}
			-1 &0 &0\\[4pt]
			0 &1 &0\\[4pt]
			0 &0 &-1
		\end{pmatrix}.
	\end{equation}
Let $\gamma_1$ and $\gamma_2$ be two loops in $\overline{M_k}$ whose projections to the $z$-plane are illustrated in Figure \ref{fig:homologicalbasis} as in \cite{LM}. The winding number of $\gamma_1$ around $0, -1/r$ is 1. The winding number of $\gamma_2$ around 0 is 1, around $r$ is $-1$. The set 
	\[
		\{ (\kappa_j)^{m}_{*}(\gamma_l) \; | \; j, l \in\{1, 2\}, m \in \{1,\dots, k+1\} \}
	\]
contains a homology basis of $\overline{M_k}$.

	\begin{figure}[htbp]
	\centering
	\includegraphics[width=9.0cm]{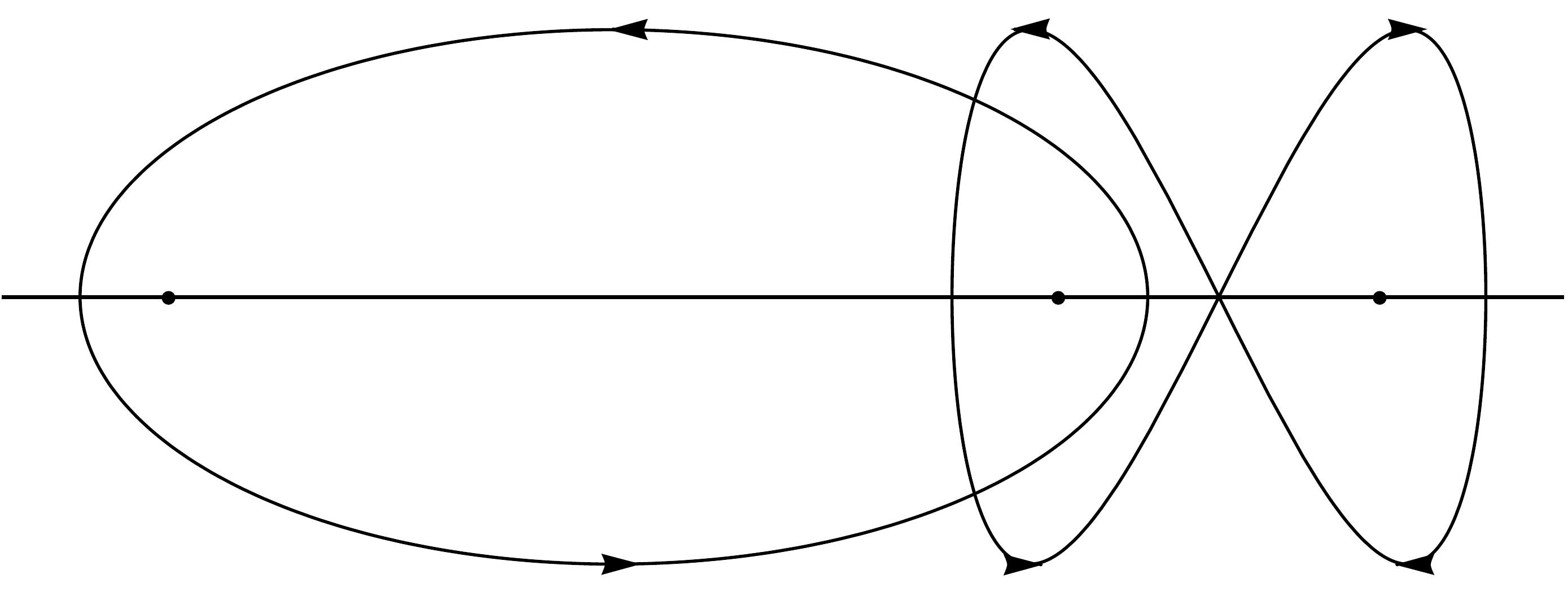}
 		\begin{picture}(0,0)
 		\put(-230,40){\makebox(0,0)[cc]{\footnotesize$-1/r$}}
 		\put(-90,40){\makebox(0,0)[cc]{\footnotesize$0$}}
 		\put(-30,40){\makebox(0,0)[cc]{\footnotesize$r$}}
 		\put(-152,100){\makebox(0,0)[cc]{\footnotesize$\gamma_1$}}
 		\put(-54,100){\makebox(0,0)[cc]{\footnotesize$\gamma_2$}}
 		\end{picture}
	\caption{Projection to the $z$-plane of $\gamma_1, \gamma_2 \in H_1(\overline{M_k}, \Z)$.}\label{fig:homologicalbasis}
	\end{figure}
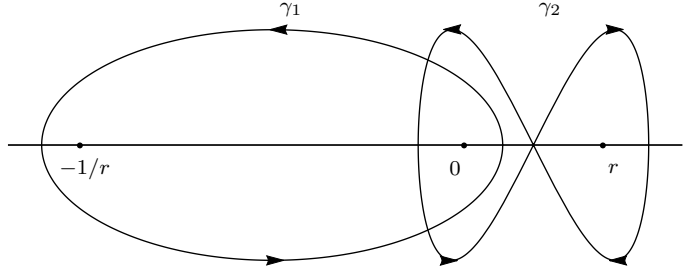

Moreover, a straightforward computation gives that
	\[
		I_{*}(\gamma_1) = \gamma_1, \qquad I_{*}(\gamma_2) = \gamma_1 - \gamma_2 + (\kappa_1)^k_{*}(\gamma_1).
	\]
Therefore,  (\ref{eq:nonori_period}) is equivalent to 
	\[
		\oint_{\gamma_1} \phi_j = 0 \qquad (j = 1, 2).
	\]
In other words, $f$ is well-defined on $M_k$ if and only if
	\[
		\oint_{\gamma_1} (1 + g^2)\eta = \oint_{\gamma_1}(1 - g^2)\eta = 0
	\]
holds, that is to say, 
	\[
		\oint_{\gamma_1} \eta = \oint_{\gamma_1}g^2\eta = 0
	\]
holds. However,
	\[
		\oint_{\gamma_1} \eta = \oint_{I_{*}(\gamma_1)} I^{*}\eta =\oint_{\gamma_1}\overline{g^2\eta},
	\]
and hence $f$ is well-defined on $M_k$ if and only if
	\[
		\oint_{\gamma_1} g^2\eta = \oint_{\gamma_1} \frac{w^k(z+1)^2}{z^3}dz = 0.
	\]
We take the exact 1-form $dF$, where
	\begin{align*}
		F &= \frac{(k+1)(z-r)(2rz^2-((k+1)r^2-2(k+2)r+k)z+r)}{(k+2)rwz}.
	\end{align*}
Then we have
	\[
		\dfrac{w^k(z+1)^2}{z^3}dz + dF = \dfrac{a_0+a_1rz}{(k+2)rw}dz,
	\]
	where
	\[
		a_0 = -(k+1)(k+2)r^2+2k(k+2)r-k(k-1), \qquad a_1 = 2(2k+1).
	\]
Since the right-hand side is a holomorphic differential on $M_k$, the loop $\gamma_1$ can be collapsed over the interval $[-1/r, 0]$ by the Stokes theorem. 
Thus $f$ is well-defined on $M_k$ if and only if
	\begin{equation}\label{eq:p1}
		p(r) :=  \int_{-1/r}^{0} \frac{a_0 + a_1 r z}{r |w|} dz = 0.
	\end{equation}
Changing the coordinate $z = -t/r$, the equation (\ref{eq:p1}) becomes
	\begin{equation}\label{eq:p2}
		p(r) = |r|^{\frac{-2k}{k+1}} \int_{0}^{1}\frac{a_0 - a_1 t}{W(t)} dt, 
	 \end{equation}
where
	\begin{equation}\label{eq:w}
		W(t) = \sqrt[k+1]{\frac{t(t+r^2)}{1-t}}.
	\end{equation}
Straightforward computations give that
	\begin{align*}
		&p_+(0):=\lim_{r>0, r\to 0} p(r) = -\infty, \\
		&p(+\infty) := \lim_{r \to \infty} p(r) = - (k+1)(k+2) \int_{0}^{1}\sqrt[k+1]{\dfrac{1-t}{t}} dt < 0, \\
		&p_-(0):=\lim_{r<0, r \to 0} p(r) = -\infty, \\
		&p(-\infty) := \lim_{r \to -\infty} p(r) = - (k+1)(k+2) \int_{0}^{1}\sqrt[k+1]{\dfrac{1-t}{t}} dt <0.
	\end{align*}
	
We want to prove that $p(k/(k+1)) >0$ for any $k \geq 2$. It is clear that
	\[
		p(k/(k+1)) = \frac{(k + 1)(2k + 1)}{k} \int_{0}^{1}\frac{1 - 2(1 + 1/k)t}{\sqrt[k + 1]{t(1+(1+1/k)^2t)/(1-t)}}dt.
	\]
Define the function $q(k)$ by
	\[
		q (k) := \int_{0}^{1} \frac{1-2(1 + 1/k)t}{V(t)}dt, 
	\]
where
	\[
		V(t)=\sqrt[k+1]{\frac{t(1 + (1 + 1/k)^2t)}{1-t}}.
	\]
A straightforward computation gives that
	\[
		 \frac{1-2(1 + 1/k)t}{V(t)}dt - d\left(\frac{(k+1)t(1-t)}{kV(t)}\right) = \frac{(k+1)^2 (1-t)t}{k\left((k+1)^2t+k^2\right)V(t)}dt.
	\]
Hence we see that
	\[
		q(k)=\int_{0}^{1} \frac{1-2(1 + 1/k)t}{V(t)}dt=\int_{0}^{1} \frac{(k+1)^2 (1-t)t}{k\left((k+1)^2t+k^2\right)V(t)}dt > 0.
	\]
Therefore, we get $p(k/(k + 1)) > 0$ for any $k \geq 2$. It is clear that
	\[
		p(1) = -2 \int_{0}^{1} \frac{(2k+1)t-k+1}{U(t)}dt,
	\]
where 
	\[
		U(t) = \sqrt[k+1]{\frac{t(1+t)}{1-t}}.
	\]
Since 
	\[
		\frac{(2k+1)t-k+1}{U(t)}dt + d\left( \frac{(k^2-1)t(1-t)}{k U(t)}\right) = \frac{t((2k+1)t+3)}{k(1+t)U(t)}dt,
	\]
it is easy to check 
	\[
		p(1) = -2 \int_{0}^{1} \frac{(2k+1)t-k+1}{U(t)}dt = -2 \int_{0}^{1} \frac{t((2k+1)t+3)}{k(1+t)U(t)}dt < 0.
	\]
As a consequence, $p$ has at least two roots in $(0, 1)$.

Let us show that $p$ has exactly two real roots on $\R\bash\{0, 1\}$. 

Define $A_0, A_1 : \R \to \R_{>0}$ by
	\[
		A_0(r) = \int_{0}^{1} \frac{dt}{W(t)}, \qquad A_1(r) = \int_{0}^{1} \frac{t}{W(t)}dt,
	\]
where $W(t)$ is given in (\ref{eq:w}). 
It is clear that (\ref{eq:p2}) becomes
	\[
		p(r) = |r|^{\frac{-2k}{k+1}}(a_0 A_0 - a_1 A_1).
	\]
Suppose that $r_0 \in \R$ satisfies $p(r_0) = 0$, then
	\[
		A_1(r_0)=\frac{a_0(r_0)}{a_1(r_0)}A_0(r_0) = -\frac{(k+1)(k+2)r_0^2-2k(k+2)r_0+k(k-1)}{2(2k+1)}A_0(r_0),
	\]
hence we see that
	\begin{equation}\label{eq:r0}
		0 \leq \frac{k(k+2)-\sqrt{k(k+2)(2k+1)}}{(k+1)(k+2)} < r_0 < \frac{k(k+2)+\sqrt{k(k+2)(2k+1)}}{(k+1)(k+2)}.
	\end{equation}
Therefore, we can conclude $r_0 \in (0, \infty)$. On the other hand, 
	\[
		A_0'(r) = -\dfrac{2r}{k+1}\int_{0}^{1}\dfrac{dt}{(t+r^2)W(t)}, \qquad A_1'(r) = -\dfrac{2r}{k+1}\int_{0}^{1}\dfrac{t}{(t+r^2)W(t)}dt.
	\]
Since 
	\begin{align*}
		& -\dfrac{2r}{(k+1)(t+r^2)W(t)}dt + d\left(\dfrac{2t(1-t)}{r(r^2+1)W(t)}\right) \\
		&\phantom{-}= \frac{2\left(-1 + k - r^2\right)}{(k+1) \left(r^3+r\right)W(t)}dt-\frac{2 (2 k+1)t}{(k+1) \left(r^3+r\right)W(t)}dt
	\end{align*}
and
	\begin{align*}
		& -\dfrac{2rt}{(k+1)(t+r^2)W(t)}dt - d\left(\dfrac{2rt(1-t)}{(r^2+1)W(t)}\right) \\
		&\phantom{-}= -\frac{2kr}{(k+1)(r^2+1)W(t)}dt+\frac{2r(2k +1)t}{(k+1)(r^2+1)W(t)}dt,
	\end{align*}
it is straightforward to see that
	\begin{align*}
		& p'(r)=b_0A_0+b_1A_1 \qquad (r > 0), \\
		& p''(r)=c_0A_0+c_1A_1 \qquad (r > 0),
	\end{align*}
where
	\begin{align*}
		b_0 &= \frac{2 k r^{\frac{-3k-1}{k+1}} \left(-(k+1)(k+2)r^3+(k-1)r^2+(k+2)(k-1)r+k-1\right)}{(k+1) \left(r^2+1\right)}, \\
		b_1 &= \frac{2 k (2 k+1) r^{\frac{-3k-1}{k+1}} \left((k+1)r^2-2(k+2)r+k+1\right)}{(k+1) \left(r^2+1\right)},\\
		c_0 &= \frac{2 k r^{\frac{-4k-2}{k+1}} \left(2 (k+1)^2 (k+2) r^3-(k+1) (4 k^2+5k-3)r^2\right)}{(k+1)^2 \left(r^2+1\right)}\\
		&\phantom{=}\quad +\frac{2 k r^{\frac{-4k-2}{k+1}} \left(-2 (k^2+k-2) r-(k+3)(k-1)\right)}{(k+1)^2 \left(r^2+1\right)},\\
		c_1 &= \frac{2 k (2 k+1) r^{\frac{-4k-2}{k+1}} \left((k+1)^2 r^2+2(k+1)(k+2)r-(3k^2+6k-1)\right)}{(k+1)^2 \left(r^2+1\right)}.
	\end{align*}
If $p(r_0)=0$, then $A_1(r_0) = (a_0(r_0)/a_1(r_0))A_0(r_0)$, and so
	\begin{align}
		p'(r_0) &= \frac{k (k+2) r_0^{-\frac{3k+1}{k+1}}Q_1(r_0)}{(1 + k)(r_0^2 + 1)}A_0(r_0),\\
		p''(r_0) &= \frac{k(k+2)r_0^{\frac{-4k-2}{k+1}}Q_2(r_0)}{(k+1) \left(r_0^2+1\right)} A_0(r_0),
	\end{align}
where
	\begin{align*}
		Q_1(r) &=-(k+1)^2r^4+2(k+1)(2k+1)r^3\\
		&\phantom{=} \quad -2(k+1)(3k+1)r^2+2(2k-1)(k+1)r-(k-1)^2,\\
		Q_2(r) &= -(k+1)^2 r^4+2 (k+1) (3 k+1) r^2-4 (k+1) (2 k-1) r+3 (k-1)^2.
	\end{align*}
Since $Q_1''(r) < 0$ for all $r \in \R$, $Q_1$ has at most two real roots. Straightforward computations give that
	\begin{align*}
		&Q_1(0) = -(k-1)^2 < 0, \\
		&Q_1(r_0^{-}) = \frac{4(k^2+k+1)(-1-2k+\sqrt{k(k+2)(2k+1)})}{(k+1)(k+2)^2} > 0, \\
		&Q_1\left(\frac{k}{k+1}\right) = -\frac{1+2k}{(k+1)^2} < 0,
	\end{align*}
where
	\begin{equation}\label{eq:r0-}
		r_{0}^{-} := \frac{k(k+2)-\sqrt{k(k+2)(2k+1)}}{(k+1)(k+2)}.
	\end{equation}
Hence the intermediate value theorem yields the existence of the exactly two real numbers $s_0 \in (0, r_0^{-})$ and $s_1 \in (r_0^{-}, k/(k + 1))$ such that $Q_1(s_0) = Q_1(s_1)=0$.
Moreover, $Q_1|_{(0, s_0)} < 0$, $Q_1|_{(s_0, s_1)} > 0$, and $Q_1|_{(s_1, \infty)} < 0$ hold.
As $A_0(r_0) > 0$, the sign of $p'(r_0)$ coincides with that of $Q_1(r_0)$.
Recall that $0 \leq r_{0}^{-} < r_{0}$ by (\ref{eq:r0}) and (\ref{eq:r0-}).
Since $p'(r_0) > 0$ while $r_0 \in (r_0^{-}, s_1)$, $p$ has at most one root lying in $(r_0^{-}, s_1)$.
Similarly, since $p'(r_0) < 0$ while $r_0 \in (s_1, \infty)$, $p$ has at most one root lying in $(s_1, \infty)$.

It only remains to prove $r_0 \neq s_1$ or $Q_2(s_1) < 0$ holds. 
In the case of $k=2$, 
	\begin{align*}
		Q_1(r) & = -9 r^4+30 r^3-42 r^2+18 r-1,\\
		Q_2(r) & = -9 r^4+42 r^2-36 r+3. 
	\end{align*}
A direct computation shows that 
	\begin{align*}
		s_1 & = \frac{5}{6} -\frac{1}{6} \sqrt{A}+\frac{1}{6} \sqrt{-A+\frac{62}{\sqrt{A}}-9} \approx 0.599176, \quad Q_2(s_1) \approx -4.65184 < 0,
	\end{align*}
where $A = 2 \sqrt[3]{3 \sqrt{197}+46}+2 \sqrt[3]{46-3 \sqrt{197}}-3$. 

In the case of $k\geq3$, the discriminant $D(Q_2)$ of $Q_2$ satisfies
	\[
		D(Q_2)=-256 (k+1)^6 \left(432 k^4-1448 k^3+1011 k^2-162 k+19\right). \\
	\]
A straightforward computation gives that
	\begin{align*}
		&432 k^4-1448 k^3+1011 k^2-162 k+19 \\
		&\quad = 432 (k-3)^4+3736 (k-3)^3+11307 (k-3)^2+13464 (k-3)+4528.
	\end{align*}
Therefore $D(Q_2) < 0$ and hence $Q_2$ has two real roots and two complex conjugate roots.
On the other hand, since
\begin{align*}
		Q_2(0) & = 3(k-1)^2 > 0, \\
		Q_2\left(\frac{k}{k+1}\right) & = 4 - 4k - \frac{1}{(k+1)^2} < 0,  \\
		Q_2\left(r_0^{-}\right) & = \dfrac{-4(k^5+2k^4-3k^3-7k^2-8k-3)}{(k+1)^2(k+2)^2} \\
		& \phantom{=} + \dfrac{4(k^3-2k^2-3k-2)\sqrt{k(k+2)(2k+1)}}{(k+1)^2(k+2)^2} < 0,
\end{align*}
$Q_2$ cannot have a real root in ($r_0^{-}, k/(k+1)$), and hence $Q_2(s_1) < 0$.

This proves that $p$ has exactly two real roots $r_1$ and $r_2$ on $\R \bash \{0,1\}$.
\end{proof}

By the above proof, we see that both $r_1$ and $r_2$ satisfy $r_1, r_2 \in (0, 1)$.
The values $r_1$ and $r_2$ can be estimated using the Mathematica software. 
See Figure $\ref{fig:period}$ for $k=2, 3, 4$.

\begin{figure}[htbp]
\centering
	\begin{tabular}{c}
		\begin{minipage}{0.31\hsize}
			\centering
			\includegraphics[width=3.5cm]{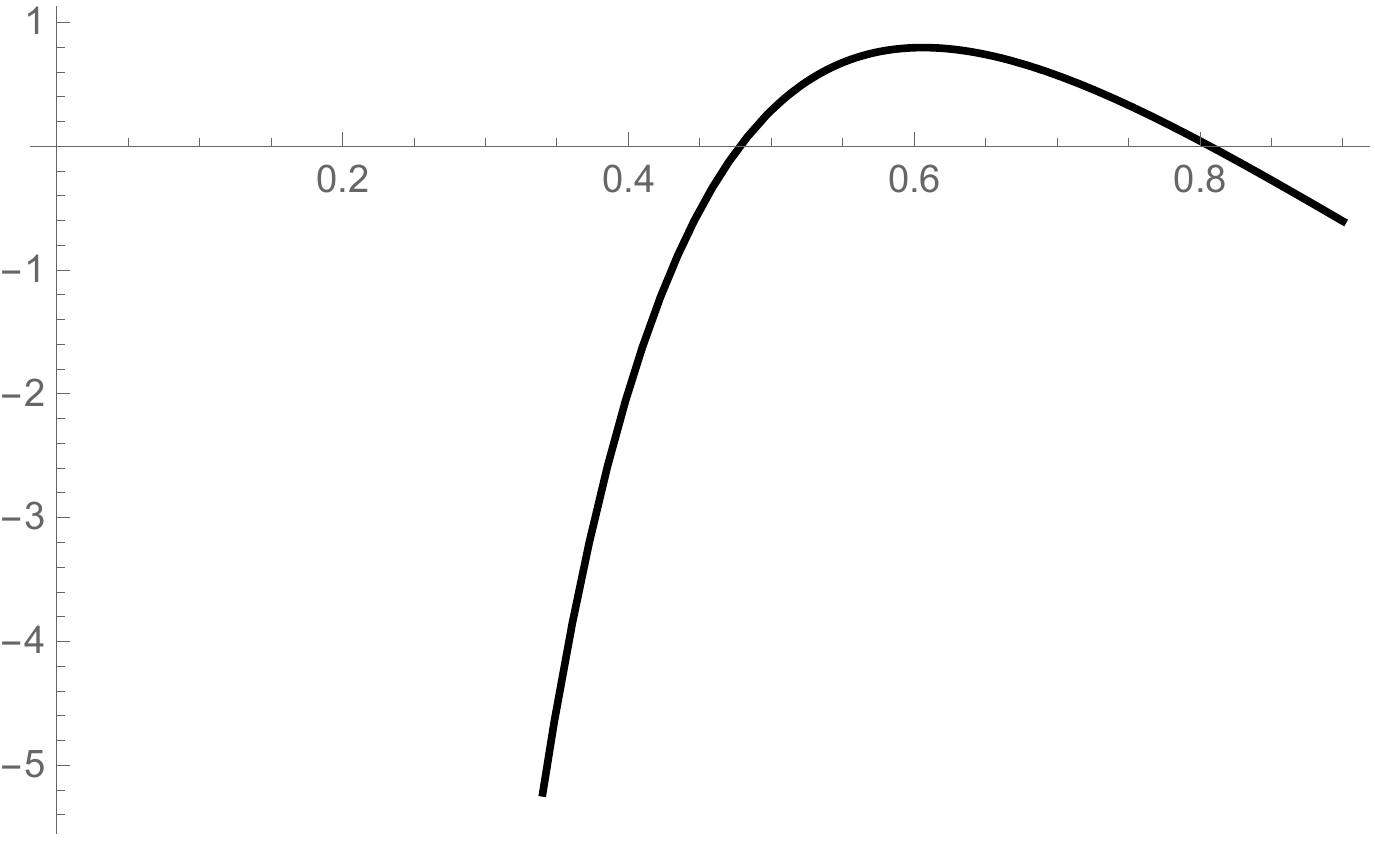}
		\end{minipage}
		
		\begin{minipage}{0.31\hsize}
			\centering
			\includegraphics[width=3.5cm]{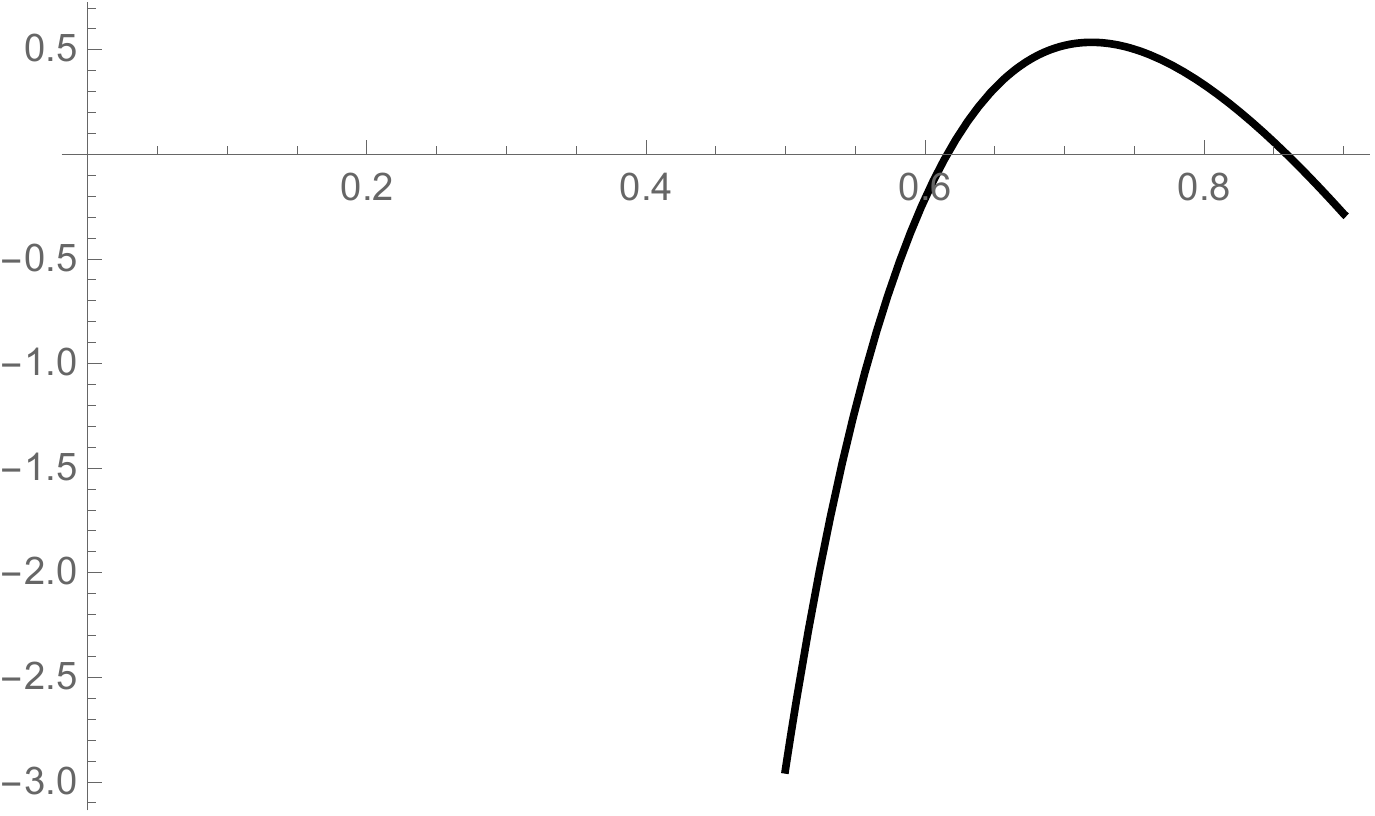}
		\end{minipage}
		
		\begin{minipage}{0.31\hsize}
			\centering
			\includegraphics[width=3.5cm]{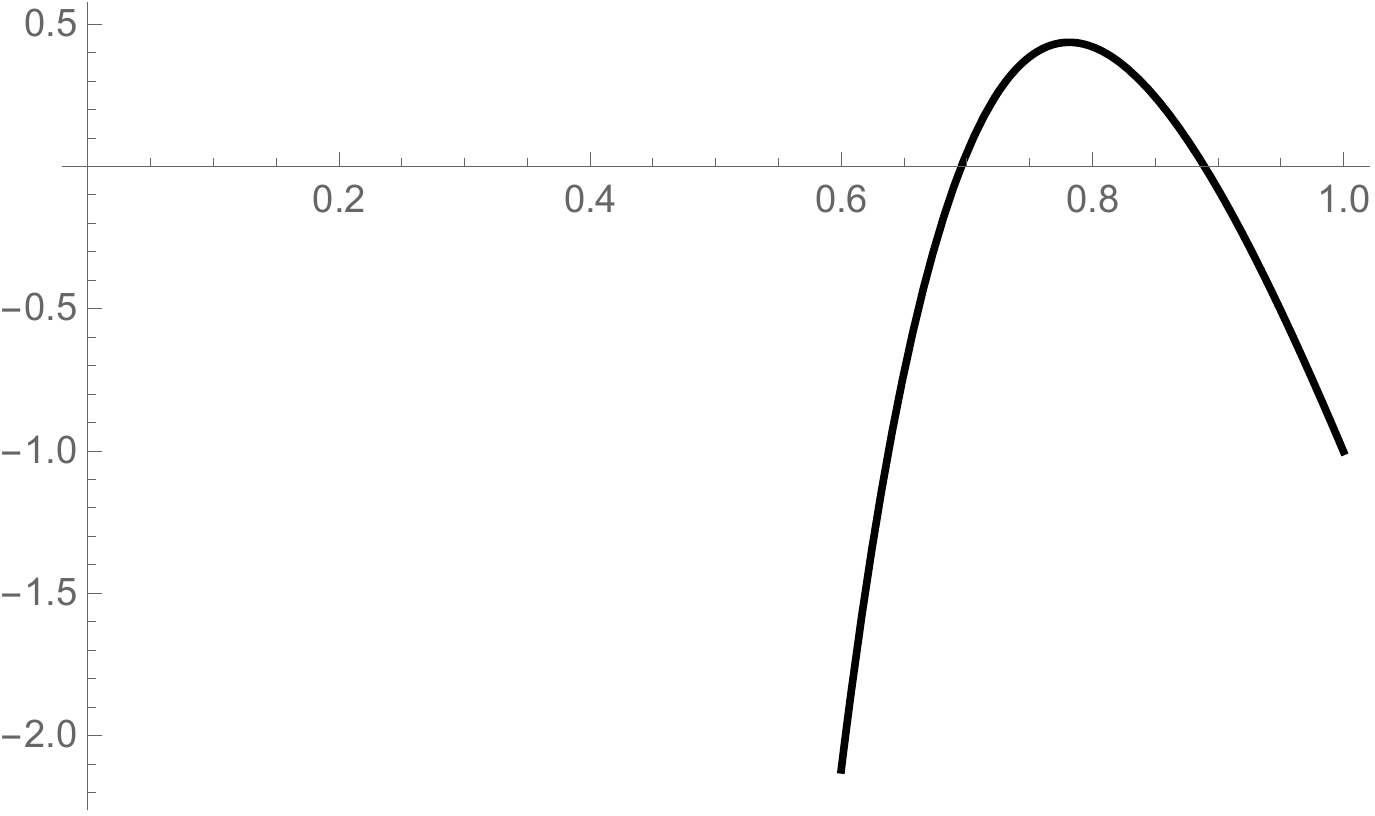}
		\end{minipage}
		\\
		\begin{minipage}{0.31\hsize}
		\centering
		$k=2$
		\end{minipage}
		
		\begin{minipage}{0.31\hsize}
		\centering
		$k=3$
		\end{minipage}
		
		\begin{minipage}{0.31\hsize}
		\centering
		$k=4$
		\end{minipage}
	\end{tabular}
\caption{The period function $p(r)$.}\label{fig:period}
\end{figure}

\begin{Rem}
For $k=2$ and $r=r_1\approx 0.478169$, we see that the singular set of $f'$ consists of 12 swallowtails, 15 cuspidal cross caps, and cuspidal edges by applying the criteria in \cite{UY} and \cite{FSUY} numerically. See Figure \ref{fig:singular}.
\begin{figure}[htbp]
	\centering
	\includegraphics[scale=0.22]{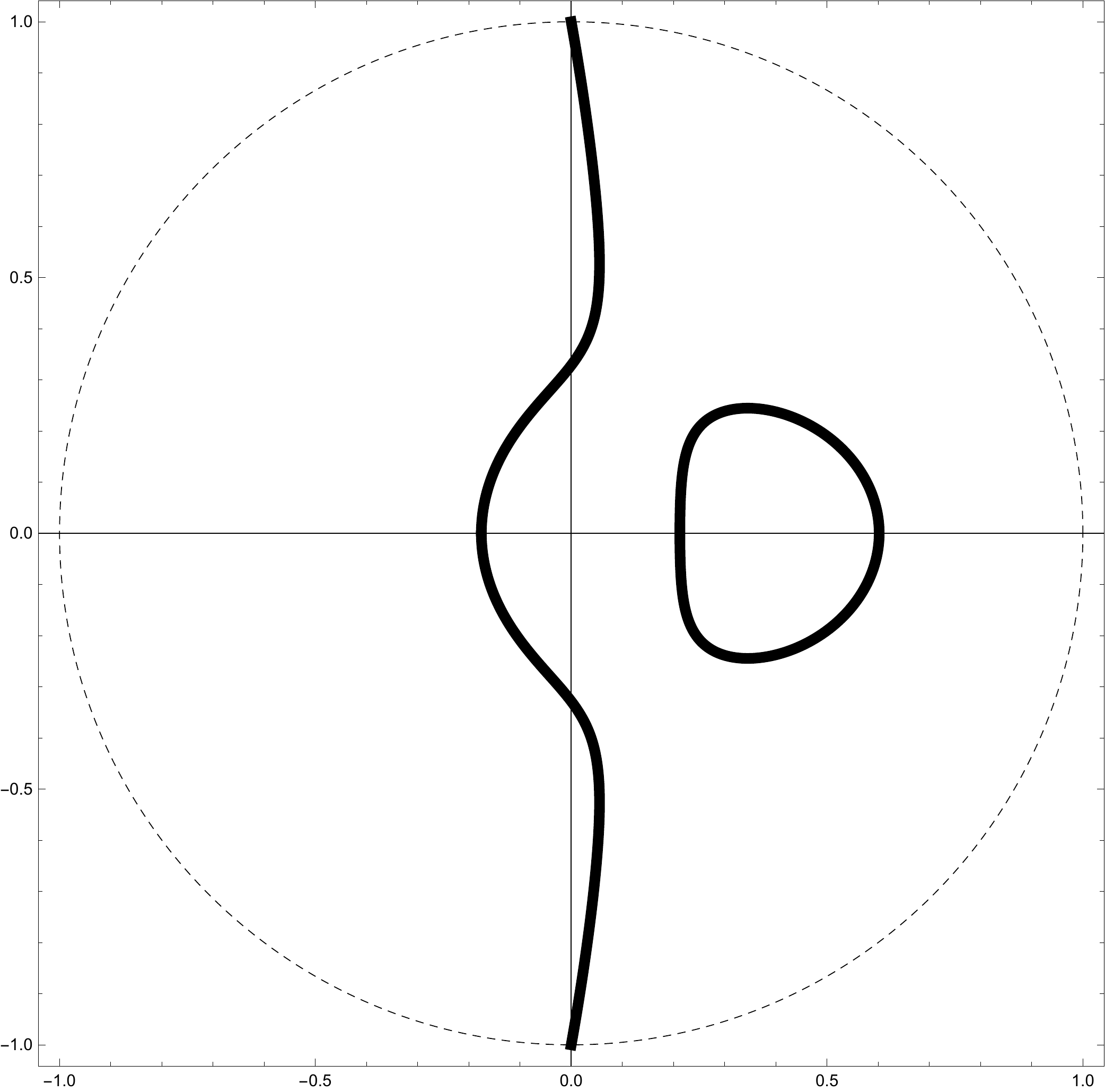}
 		\begin{picture}(0,0)
		\put(-31,67){\makebox(0,0)[cc]{\footnotesize $\bigcirc$}}
		\put(-54,67){\makebox(0,0)[cc]{\footnotesize $\bigcirc$}}
		\put(-78,67){\makebox(0,0)[cc]{\footnotesize $\bigcirc$}}
		\put(-37,79){\makebox(0,0)[cc]{\footnotesize $\bigcirc$}}
		\put(-37,55){\makebox(0,0)[cc]{\footnotesize $\bigcirc$}}
		\put(-73,55){\makebox(0,0)[cc]{\footnotesize $\bigtriangleup$}}
		\put(-74,80){\makebox(0,0)[cc]{\footnotesize $\bigtriangleup$}}
		\put(-54,76){\makebox(0,0)[cc]{\footnotesize $\bigtriangleup$}}
		\put(-54,60){\makebox(0,0)[cc]{\footnotesize $\bigtriangleup$}}
  		\end{picture}
	\caption{The singular set of $f'$ ($k=2, r=r_1$) in the unit disk of the projection to the $z$-plane of $M_k$.
	The thick curves indicate the singular points.
	$\bigcirc$ indicates a cuspidal cross cap and $\bigtriangleup$ indicates a swallowtail.
	The other singularities are cuspidal edges.}\label{fig:singular}
	\end{figure}
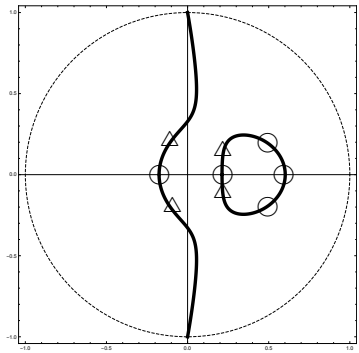
\end{Rem}

\section{Symmetry}\label{sec:sym}
We consider the symmetry of the surfaces constructed in Theorem \ref{th:ex}.
Let $f' : M' \to \L^3$ be the nonorientable complete maxface constructed in Theorem \ref{th:ex} with the maxface $f = f' \circ \pi : M_k \to \L^3$.
An isometry $S : M_k \to M_k$ is said to be a symmetry of $f$ if there exists a Lorentzian isometry $\tilde{S} : \L^3 \to \L^3$ such that $f \circ S = \tilde{S} \circ f$.
We denote by Sym($M_k$) the symmetry group of $M_k$. 

\begin{Prop}
The surface $f:M_k\to\mathbb{L}^3$ intersects the $(x_1,x_2)$-plane in $(k+1)$ straight lines which meet at equal angles at the origin,
and the symmetry group {\rm{Sym}}$(M_k)$ of the surface $f : M_k \to \L^3$ is equal to the dihedral group $D(k+1)$ generated by $\kappa_1$ and $\kappa_2$ in $(\ref{eq:k1k2})$.
\end{Prop}

\begin{proof}
We will determine Sym$(M_k)$.
The elements of  the dihedral group $D(k+1)$ generated by $\kappa_1, \kappa_2$ in (\ref{eq:k1k2}) can be described as follows:
\begin{itemize}
	\item $K_1^j$ is rotation by $-2\pi j/(k+1)$ about the $x_3$-axis, $j=1,\dots, k+1$,
	\item  $K_1^{j}K_2$ is the reflectional symmetry with respect to the straight line $x_2 = -\tan ((k-1)j\pi/(k+1))x_1, x_3=0, j=1,\dots,k+1$,
\end{itemize}
where $K_1, K_2$ are given in (\ref{eq:K}).
Thus $D(k+1) \subset {\rm{Sym}}(M_k)$.

Next, we observe the asymptotic behavior of the ends.
At $(z,w)=(0,0)$, $w$ is a local coordinate for the Riemann surface $\overline{M_k}$, and then
\[
	z = z(w) = w^{k+1}\left(-\frac{1}{r}+\O(w^{k+1})\right).
\]
We have
\begin{equation}\label{eq:exp}
\begin{aligned}
	g &= \frac{\alpha_0}{w}+\O(w^{k}), & \eta &= \left(\frac{i \alpha_1}{w^{k+1}}+\O(1)\right)dw, \\
	g\eta &= \left( \frac{i \alpha_2}{w^{k+2}}+\O(w^{k})\right)dw, & g^2\eta &= \left( \frac{i \alpha_3}{w^{k+3}}+\frac{\alpha_4}{w^2}+\O(w^{k-1})\right)dw,
\end{aligned}
\end{equation}
where $\alpha_0, \alpha_1, \alpha_2, \alpha_3 \in \R$ and $\alpha_4 \in \C$ are constants.
Note that $g\eta$ is exact (see (\ref{eq:geta})).
This implies that the end of $f$ is asymptotic to the end of the surface determined by the Weierstrass data
\[
	(g_0, \eta_0) = \left( \frac{1}{z}, \frac{i}{z^{k+1}}dz\right) \text{\quad on \quad $\mathbb{D}_{\varepsilon}^{*}=\{z \in \C \: |\:  0< |z|<\varepsilon \}$}.
\]
Moreover, let $f_0 : \mathbb{D}_{\varepsilon}^{*} \to \L^3$ be the end of maxface with the Weierstrass data $(g_0, \eta_0)$.
Then
\begin{align*}
	f_0 &= \left(\frac{-\sin(k\theta)}{\rho^kk}-\frac{\sin(k+2)\theta}{\rho^{k+2}(k+2)}, \frac{\cos(k\theta)}{\rho^{k}k}-\frac{\cos(k+2)\theta}{\rho^{k+2}(k+2)}, \frac{-2\sin(k+1)\theta}{\rho^{k+1}(k+1)}\right)\\
	& = (x_1, x_2, x_3),
\end{align*}
where $z = \rho e^{i\theta}$.
Since
\[
	x_1^2+x_2^2-x_3^2 = \frac{1}{(k+2)^2\rho^{2(k+2)}}+\frac{1}{\rho^{2k+2}}\left(\frac{-2\cos2\theta}{k(k+2)}+\frac{4\sin^2(k+1)\theta}{(k+1)^2}+\O(\rho^2)\right),
\]
we have 
 \[
 	\mathcal{D} = f_0^{-1}(\{(x_1, x_2, x_3)\in\L^3 \: | \: x_1^2+x_2^2-x_3^2 \geq R\} ) \cup \{ 0 \} \subset \overline{\mathbb{D}_{\varepsilon}}
 \]
is a conformal disk if $R$ is large enough.
Since Sym($M_k$) leaves $\mathcal{D}$ invariant and fixes $0$, the group $\{ S|_{\mathcal{D}
 } \: | \: S \in {\rm{Sym}}(M_k)\}$ is included in $D(k+1)$.
Therefore ${\rm{Sym}}(M_k) = D(k+1)$.
\end{proof}
\section*{Acknowledgements}
The authors would like to thank Professor Francisco J. L\'{o}pez for valuable comments and suggestions.

\bigskip
%%%%%%%%%%%% Authors addresses %%%%%%%%%%%%
{\small \noindent Shoichi Fujimori \\ Department of Mathematics \\ Hiroshima University \\ Higashihiroshima, Hiroshima 739-8526, Japan \\ {\itshape E-mail address}\/: \href{mailto:fujimori@hiroshima-u.ac.jp}{fujimori@hiroshima-u.ac.jp}
\par\vskip4ex
\noindent
Shin Kaneda \\
Department of Mathematics \\
Hiroshima University \\
Higashihiroshima, Hiroshima 739-8526, Japan \\ {\itshape E-mail address}\/: \href{mailto:shin-kaneda@hiroshima-u.ac.jp}{shin-kaneda@hiroshima-u.ac.jp}
}

\end{document}